\documentclass[final]{siamltex}
\usepackage{lipsum}
\usepackage{amsfonts}
\usepackage{graphicx}
\usepackage{epstopdf}
\usepackage{mathrsfs}
\usepackage{amsmath}
\usepackage{amssymb}
\usepackage{subfigure}
\usepackage{algorithmic}
\usepackage{multirow}
\usepackage{xcolor}
\usepackage{latexsym,array,extarrows}
\usepackage{float}
\usepackage{amsopn}

\ifpdf
\DeclareGraphicsExtensions{.eps,.pdf,.png,.jpg}
\else
\DeclareGraphicsExtensions{.eps}
\fi

\makeatletter
\def\@textbottom{\vskip \z@ \@plus 200pt}
\let\@texttop\relax
\makeatother
\numberwithin{theorem}{section}

\newtheorem{property}{Property}
\newtheorem{rem}{Remark}

\newcommand{\diago}{\mathrm{diag}}

\title{A $\tau$-approximation based preconditioning technique for space fractional diffusion equation with non-separable variable coefficients
\thanks{The corresponding author. The work of Xuelei Lin was supported by research grants: 2021M702281 from   China Postdoctoral Science Foundation; HA45001143, HA11409084 two start-up Grants from Harbin Institute of Technology, Shenzhen.}}
\author{Xue-lei Lin\thanks{School of Science, Harbin Institute of Technology, Shenzhen 518055, China.
		(e-mail:hxuellin@gmail.com).}
	\and
	Michael K. Ng\thanks{Department of
		Mathematics, Hong Kong Baptist University, Kowloon Tong, Hong Kong (e-mail:michael-ng@hkbu.edu.hk).}
}

\begin{document}
	\maketitle	
	\begin{abstract}
		In this paper, we study a $\tau$-matrix approximation based preconditioner for  the linear systems arising from discretization of unsteady state Riesz space fractional diffusion equation with non-separable variable coefficients. The structure of coefficient matrices of the linear systems is identity plus summation of diagonal-times-multilevel-Toeplitz matrices. In our preconditioning technique, the diagonal matrices are approximated by scalar identity matrices and the Toeplitz matrices which are approximated by $\tau$-matrices (a type of matrices diagonalizable by discrete sine transforms). The proposed preconditioner is fast invertible through the fast sine transform (FST) algorithm. Theoretically, we show that the GMRES solver for the preconditioned systems has an optimal convergence rate (a convergence rate independent of discretization stepsizes). 
To the best of our knowledge, this is the first preconditioning method with the optimal convergence rate for the variable-coefficients space fractional diffusion equation. Numerical results are reported to demonstrate the efficiency of the proposed method.
	\end{abstract}
	
	\begin{keywords}
		optimal convergence; preconditioners; variable coefficients; space-fractional diffusion equations;
		Krylov subspace methods
	\end{keywords}	
	\begin{AMS}
		65B99; 65M22; 65F08; 65F10
	\end{AMS}

	\section{Introduction}\label{introduction}
	We  firstly consider the two-dimension initial-boundary value problem of space-fractional diffusion equation (SFDE) \cite{chendengwu2013} (the multi-dimensional case will be  considered in Section 4):
	\begin{align}
	&\frac{\partial u(x,y,t)}{\partial t}=d(x,y)\frac{\partial^{\alpha}u(x,y,t)}{\partial|x|^{\alpha}}+e(x,y)\frac{\partial^{\beta}u(x,y,t)}{\partial|y|^{\beta}}+f(x,y,t),\notag\\
	&\qquad\qquad\qquad\qquad\qquad\qquad\qquad\qquad\qquad\qquad\qquad\qquad(x,y,t)\in\Omega\times(0,T],\label{rsdiffusioneq}\\
	&u(x,y,t)=0,\qquad\qquad\qquad\qquad\qquad\qquad\qquad \qquad\qquad(x,y,t)\in\partial\Omega\times[0,T],\label{drcheltboundary}\\ 
	&u(x,y,0) =\psi(x,y),\qquad\qquad\qquad\qquad\qquad\quad\quad\qquad\quad~ (x,y)\in\bar{\Omega},\label{initialcondition}
	\end{align}
	$\Omega=(l_1,r_1)\times(l_2,r_2)$ is an open rectangle; $\partial\Omega$ denotes the boundary of $\Omega$;
	where the given  coefficient functions $ \hat{d}\geq d \geq \check{d}>0$,  $ \hat{e}\geq e \geq \check{e}>0$ for some positive constants $\hat{d},\check{d},\hat{e},\check{e}>0$;
	$f$ and $\psi$ are both given; $\frac{\partial^{\alpha}u}{\partial|x|^{\alpha}}$ ($\frac{\partial^{\alpha}u}{\partial|y|^{\beta}}$, respectively) is the Riesz fractional derivative of order $\alpha\in(1,2)$ ($\beta\in(1,2)$, respectively) with respect to $x$ ($y$, respectively) defined as follows whose definition is given by \cite{podlubny1999}
		\begin{align*}
		&\frac{\partial^{\alpha}u(x,y,t)}{\partial|x|^{\alpha}}:=\frac{-1}{2\cos(\alpha\pi/2)\Gamma(2-\alpha)}\frac{\partial^2}{\partial x^2}\int_{l_1}^{r_1}\frac{u(\xi,y,t)}{|x-\xi|^{\alpha-1}}d \xi,\\
		&\frac{\partial^{\beta}u(x,y,t)}{\partial|y|^{\beta}}:=\frac{-1}{2\cos(\beta\pi/2)\Gamma(2-\beta)}\frac{\partial^2}{\partial x^2}\int_{l_2}^{r_2}\frac{u(x,\xi,t)}{|y-\xi|^{\beta-1}}d \xi,
		\end{align*}
		with $\Gamma(\cdot)$ being the gamma function.

	Fractional differential equations have gained considerable attention in the last few decades due to  its applications in various fields of science and engineering, such as electrical and mechanical engineering, biology, physics, control theory, data fitting; see  \cite{podlubny1999,hilferrudolf-2000,bouchaud-1990,metzler-2000,solomon-1993,agrawal-2002}. As a class of fractional differential equations, space fractional diffusion equations have been widely and successfully used in modeling challenging phenomena such as long-range interactions, nonlocal dynamics \cite{benson2000,podlubny1999}.
	
	Since the closed-form analytical solutions of fractional diffusion equations are usually unavailable,  discretization schemes are proposed to solve fractional diffusion equations; see, for instance, \cite{zhaozsunwcao2015,wangwang2010,leislhyc2016,liuqliufguyt2015,meerschaert2006,sousaelic,tianwy2015,chendeng2014,celikduman2012}. Since the fractional differential operator is non-local, its  numerical discretization leads to dense linear systems. That means, direct solver like Gaussian elimination is time-consuming for solving the discrete fractional differential equations.
       This promotes the development of fast solvers for solving discrete   fractional differential equations.
	
	Implicit uniform-grid discretization of  \eqref{rsdiffusioneq}--\eqref{initialcondition} leads to a  dense Toeplitz-like linear system at each time-step. 
	Fortunately, the matrix-vector multiplication of the Toeplitz-like matrix can be fast implemented using fast Fourier transforms (FFTs); see, e.g., \cite{wang2012fast}. The fast matrix-vector multiplication motivates the development of iterative solvers for the Toeplitz-like linear system arising from \eqref{rsdiffusioneq}--\eqref{initialcondition}.   In \cite{lin2017multigrid}, a geometric multigrid method is proposed for solving the Toeplitz-like linear systems, which however has no convergence analysis due to the complicated iteration matrices. In \cite{hkpang2012}, an algebraic multigrid method is proposed for the Toeplitz-like linear system arising from discretization of one-dimension SFDE and the convergence of the two-grid algorithm is established under the assumption that the diffusion coefficient function is a constant. In \cite{leisun2013}, a (multilevel) circulant preconditioner is proposed for the Toeplitz-like linear system and the convergence of the preconditioned GMRES solver is established under the assumption that $d$ and $e$ are both constants. In \cite{pankeng2014}, an approximate-inverse preconditioner is proposed for the Toeplitz-like system and the convergence of GMRES solver is established under the assumption that  the ratios between temporal and spatial discretization step-sizes ($\Delta t/ \Delta x^{\alpha}$ and $\Delta t/\Delta y^{\beta}$) are all constants. In \cite{jinlinzhao2015}, a banded preconditioner is proposed for the Toeplitz-like system and the inversion of the banded preconditioner is inexactly implemented by the incomplete LU factorization. The convergence rate of the preconditioned GMRES with the banded preconditioner deteriorates as  $\Delta t/ \Delta x^{\alpha}$ or $\Delta t/\Delta y^{\beta}$ increases. In \cite{donatelli2016spectral} and \cite{barakitis2022preconditioners}, preconditioners based on spectral symbol are proposed for the Toeplitz-like system and the  spectrum distribution of the preconditioned matrix is analyzed under the assumption that $\Delta t/ \Delta x^{\alpha}$ and $\Delta t/\Delta y^{\beta}$ are all constants. To improve the aforementioned theoretical results relying on the assumption that $d$ and $e$ are constants or that $\Delta t/ \Delta x^{\alpha}$ and $\Delta t/\Delta y^{\beta}$ are constants, the authors in  \cite{lin2017splitting} proposed a (multilevel) Toeplitz preconditioner for the Toeplitz-like system and  the condition number of the preconditioned matrix is shown to be uniformly bounded by constants independent of discretization step-sizes under a weaker assumption that $d$ and $e$ are of simple structure, such as separable functions. The fast inversion of the Toeplitz preconditioner proposed in \cite{lin2017splitting} resorts to multigrid inner iterations, which is complicated to implement. In \cite{noutsos2016essential}, a $\tau$-matrix approximation based preconditioner constructed from generating function of Toeplitz is proposed for Toeplitz system. The authors in \cite{noutsos2016essential} give bounds of the spectrum of the preconditioned matrix provided that the underlying generating function is a trigonometric polynomial. Nevertheless, for Toeplitz matrices arising from SFDEs, the generating functions are not  trigonometric polynomials. That means there is no theoretical bounds provided for spectrum of preconditioned matrix when applying the preconditioner proposed in \cite{noutsos2016essential} to preconditioning the Toeplitz matrices arising from SFDEs.
	Recently, the authors in \cite{huangxin2022} proposed a $\tau$-matrix approximation based preconditioner for steady-state constant-coefficients Riesz fractional diffusion equation, which approximates the related Toeplitz matrices with $\tau$-matrices\footnote{a type of matrices that are diagonalizable by discrete sine transforms}. With the preconditioner proposed in \cite{huangxin2022}, the preconditioned matrix is shown to have a spectrum lies in $(1/2,3/2)$. Moreover, the  preconditioner proposed in \cite{huangxin2022} is fast invertible through the FSTs. The fast invertibility and the optimal convergence of the $\tau$-matrix based preconditioning technique lead to the quasi-linear complexity of PCG solver for solving the steady-state problem discussed in \cite{huangxin2022}.
	
	Motivated by the $\tau$-matrix based preconditioning technique proposed in \cite{huangxin2022}, we propose a $\tau$-preconditioner for the Toeplitz-like system arising from the variable-coefficients SFDE  \eqref{rsdiffusioneq}--\eqref{initialcondition}. The structure  of the Toeplitz-like linear system is a summation of an identity matrix and two  diagonal-times-two-level-Toeplitz matrices. In our preconditioning technique, the diagonal parts are approximated by scalar identity matrices and the multi-level Toeplitz parts are approximated by $\tau$-matrices. We call the so obtained preconditioner as $\tau$-preconditioner. The inversion of the $\tau$-preconditioner can be fast and exactly implemented by the FSTs, which is in contrast to the inversion of the Toeplitz preconditioner proposed in \cite{lin2017splitting} resorting to inexact multigrid inner iteration. More importantly, we show that the GMRES solver for the preconditioned linear system has an optimal convergence rate (a linear convergence rate independent of discretization step-sizes) with  assumption that the diffusion coefficient function $d$ ($e$, respectively) is partially Lipschitz-continuous with respect to  $x$ ($y$, respectively). Such an assumption is weaker than the assumption used in \cite{lin2017splitting}, since it does not require the separability of the diffusion coefficient functions. The proposed preconditioning technique and the theoretically results can be extended to the multi-dimension case. To the best of our knowledge, this is the first optimal preconditioning technique\footnote{an optimal preconditioning technique means iterative solver for the preconditioned system has a linear convergence rate independent of discretization step-sizes}  for the SFDE with non-separable variable-coefficients. Numerical examples are given to support the theoretical results
	and to show the efficiency of the proposed preconditioning method.
	
	The outline of this paper is as follows. In Section 2, we present the Toeplitz-like linear systems arising from discretization of \eqref{rsdiffusioneq}-\eqref{initialcondition} and review some properties of the coefficient matrices. In Section 3, we propose  the $\tau$-preconditioner for the Toeplitz-like linear system and analyze the convergence rate of GMRES solver for the preconditioned system. In Section 4, we extend the $\tau$-preconditioner and the theoretical results to the multi-dimension space fractional diffusion equation.
	In Section 5, we present numerical results to show the performance of the proposed preconditioner.
	Finally, we give concluding remarks in Section 6.
	
	\section{Discretization of \eqref{rsdiffusioneq}--\eqref{initialcondition} and the Toeplitz-like Linear Systems}
	
	In this section, we discuss the discretization of the Riesz space fractional diffusion equation \eqref{rsdiffusioneq}--\eqref{initialcondition} and present the resulting linear systems. 
	
	For positive integers $N$, $M_x$ and $M_y$, let $\Delta t=T/N$, $\Delta x=(r_1-l_1)/(M_x+1)$ and $\Delta y=(l_2-r_2)/(M_y+1)$.
	Denote the set of all positive integers and the set of all nonnegative integers by $\mathbb{N}^{+}$ and $\mathbb{N}$, respectively. For any $m,n\in\mathbb{N}$ with $m\leq n$, define the set $m\wedge n:=\{m,m+1,...,n-1,n\}$.
	Define the grid points,
	\begin{align*}
	&\{t_n|t_n=n\Delta t,~  n\in 0\wedge N\}, \quad \{x_i|x_i=l_1+i\Delta x,~i\in 0\wedge (M_x+1)\}, \\ &\{y_j|y_j=l_2+j\Delta y,~j\in 0\wedge (M_y+1)\}.
	\end{align*}
	   Denote $G_{i,j}=(x_i,y_j)$ for $(i,j)\in\mathcal{I}_h$. Denote the vector assembling all spatial grid-points by
	\begin{equation}\label{vxydef}
		{\bf V}_{x,y}=(G_{1,1},G_{2,1},...,G_{M_x,1},G_{1,2},G_{2,2},...,G_{M_x,2},......,G_{1,M_y},G_{2,M_y},...,G_{M_x,M_y})^{\rm T}.
	\end{equation}

	Define the index sets, $\hat{\mathcal{I}}_{h}=\{(i,j)| i\in 0\wedge (M_x+1),~j\in 0\wedge (M_y+1)\}$, $\mathcal{I}_h=\{(i,j)|i\in 1\wedge M_x,~j\in 1\wedge M_y\}$, $\partial\mathcal{I}_h=\hat{\mathcal{I}}_{h}\setminus \mathcal{I}_h$. Let  $\{v_{i,j}^n|(i,j)\in \hat{\mathcal{I}}_h,~  n\in 0\wedge N\}$ be a grid function. Define
	\begin{align}
	&\delta_tv_{i,j}^{n}=\frac{v_{i,j}^{n}-v_{i,j}^{n-1}}{\Delta t},~d_{i,j}=d(x_i,y_j),~e_{i,j}=e(x_i,y_j),~f_{i,j}^n=f(x_i,y_j,t_n),\notag\\
	&\delta_x^{\alpha}v_{i,j}^{n}:=-\frac{1}{\Delta x^{\alpha}}\sum\limits_{k=1}^{M_x}s_{|i-k|}^{(\alpha)}v_{k,j}^{n},\quad \delta_y^{\beta}v_{i,j}^{n}:=-\frac{1}{\Delta y^{\beta}}\sum\limits_{k=1}^{M_y}s_{|j-k|}^{(\beta)}v_{i,k}^{n},\quad, n\in 1\wedge N,~ (i,j)\in\mathcal{I}_h.\label{spatialdiscform}
	\end{align}
	$\delta_x^{\alpha}$ and $\delta_y^{\beta}$ are discretization of the fractional derivatives $\frac{\partial^{\alpha}}{\partial |x|^{\alpha}}$ and $\frac{\partial^{\beta}}{\partial |y|^{\beta}}$, respectively.
	The numbers $s_{k}^{(\gamma)}$ ($k\geq 0$, $\gamma\in(1,2)$) are  determined by specific discretization schemes, see, e.g., \cite{meerschaert2006,tianwy2015,chendeng2014,ding2017high,celikduman2012}. 
	
	Then, \eqref{rsdiffusioneq}--\eqref{initialcondition} can be discretized as follows
	\begin{align*}
	&\delta_tu_{i,j}^{n}=d_{i,j}\delta_x^{\alpha}v_{i,j}^{n}+e_{i,j}\delta_y^{\beta}v_{i,j}^{n}+f_{i,j}^n,\quad n\in 1\wedge N,~ (i,j)\in\mathcal{I}_h,\\
	&u_{i,j}^n\equiv 0,\quad n\in 1\wedge N,\quad (i,j)\in\partial\mathcal{I}_h,\\
	&u_{i,j}^0=\psi(x_i,y_j),\quad (i,j)\in\mathcal{I}_h,
	\end{align*}
    which is equivalent to the following linear systems
    \begin{align}\label{discsystem}
    &\frac{1}{\Delta t}({\bf u}^n-{\bf u}^{n-1})=-\left[\frac{1}{\Delta x^{\alpha}}{\bf D}({\bf I}_{M_y}\otimes {\bf S}_{\alpha,M_x})+\frac{1}{\Delta y^{\beta}}{\bf E}( {\bf S}_{\beta,M_y}\otimes{\bf I}_{M_x})\right]{\bf u}^{n}+{\bf f}^{n},\notag\\
    & n=1,2,...,N,
    \end{align}
     where ${\bf I}_k$ denotes $k\times k$ identity matrix,
    \begin{align}
    	&{\bf u}^{n}=(u_{1,1}^{n},u_{2,1}^{n},...,u_{M_x,1}^{n},u_{1,2}^{n},u_{2,2}^{n},...,u_{M_x,2}^{n},......,u_{1,M_y}^{n},u_{2,M_y}^{n},...,u_{M_x,M_y}^{n})^{\rm T},\notag\\
    	&{\bf D}=\diag({\bf V}_{x,y}),\quad {\bf E}=\diag(e({\bf V}_{x,y})),\quad {\bf f}^{n}=f({\bf V}_{x,y},t_n),\notag\\
    	&{\bf S}_{\gamma,M}:=\left[
    	\begin{array}
    		[c]{ccccc}
    		s_0^{(\gamma)} & s_{1}^{(\gamma)} &\ldots   &s_{M-2}^{(\gamma)}  &s_{M-1}^{(\gamma)}\\
    		s_1^{(\gamma)} & s_0^{(\gamma)}&s_{1}^{(\gamma)}   &\ldots &s_{M-2}^{(\gamma)}\\
    		\vdots&\ddots &\ddots&\ddots&\vdots\\
    		s_{M-2}^{(\gamma)} & \ldots&s_1^{(\gamma)}& s_0^{(\gamma)} &s_{1}^{(\gamma)}\\
    		s_{M-1}^{(\gamma)} & s_{M-2}^{(\gamma)}& \ldots & s_1^{(\gamma)} & s_{0}^{(\gamma)}
    	\end{array}
    	\right],\quad M\geq 1,\quad \gamma\in(1,2).\label{sgamamdef}
    \end{align}
    Here, $\otimes$ denotes the Kronecker product. Removing ${\bf u}^{n-1}$ (${\bf u}^{n}$, respectively) to right (left, respectively) hand side, \eqref{discsystem} can be equivalently rewritten as 
    \begin{equation}\label{toeplitz-likesystem}
    	{\bf A}{\bf u}^{n}={\bf b}^{n}, \quad n=1,2,...,N,
    \end{equation} 
    where ${\bf b}^{n}={\bf u}^{n-1}+\tau{\bf f}^{n}$; 
    \begin{align*}
    &{\bf A}={\bf I}_{J}+\eta_x{\bf A}_x+\eta_y{\bf A}_y,\quad \eta_x=\frac{\Delta t}{\Delta x^{\alpha}},\quad \eta_y=\frac{\Delta t}{\Delta y^{\beta}},\\
    &{\bf A}_x={\bf D}({\bf I}_{M_y}\otimes {\bf S}_{\alpha,M_x}),\quad {\bf A}_y={\bf E}( {\bf S}_{\beta,M_y}\otimes {\bf I}_{M_x})
    \end{align*}
    ${\bf A}={\bf I}_{J}+\Delta t({\bf A}_x+{\bf A}_y)$;  $J=M_xM_y$. Clearly, the unknowns in \eqref{toeplitz-likesystem} can be solved in a sequential manner: ${\bf u}^{1},{\bf u}^{2},...,{\bf u}^{N}$. ${\bf S}_{\gamma,M}$ defined in \eqref{sgamamdef} is a symmetric Toeplitz matrix\footnote{A Toeplitz matrix is a matrix whose entries are constants along its diagonals}. The matrix ${\bf A}$ is the so-called Toeplitz-like matrix. 
    
    	Define a set of sequences as
    \begin{equation*}
    	\mathcal{D}_s:=\Big\{\{w_k\}_{k\geq 0} \ \Big |  \ ||\{w_{k}\}||_{\mathcal{D}_s}:=\sup\limits_{k\geq 0}|w_{k}|(1+k)^{1+s}<+\infty\Big\},
    \end{equation*}
    for some $s>0$.

    In this paper, we adopt the numerical schemes proposed in \cite{ccelik2012crank,meerschaert2006finite,sousaelic} for evaluation of $s_{k}^{(\gamma)}$ ($k\geq 0$, $\gamma\in(1,2)$). Actually  $s_{k}^{(\gamma)}$ ($k\geq 0$, $\gamma\in(1,2)$) arising from \cite{ccelik2012crank,meerschaert2006finite,sousaelic} holds the following properties.
    \begin{property}\label{skprop}
    \begin{description}
    \item [(i)] $\{s_{k}^{(\gamma)}\}_{k\geq 0}\in\mathcal{D}_{\gamma}$, $\forall \gamma\in(1,2)$;
    \item[(ii)]$s_0^{(\gamma)}>0$, $s_{k}^{(\gamma)}\leq 0$ for $k\geq 1$;
    \item[(iii)] $\inf\limits_{m\geq 1}{(m+1)^{\beta}}\left(s_0^{(\gamma)}+2\sum\limits_{k=1}^{m-1}s_k^{(\gamma)}\right)>0$;
    \item[(iv)]$s_{k}^{(\gamma)}\leq s_{k+1}^{(\gamma)}$ for $k\geq 1$.
    \end{description}
    \end{property}
    Property ${\bf (i)}$ of these schemes has been verified in \cite{lin2017splitting}. We verify Property ${\bf (ii)}$--${\bf (iv)}$ of these schemes in the Appendices \ref{schm1verifysec}--\ref{schm3verifysec}.
    
    \section{The $\tau$-preconditioner for the Toeplitz-like matrix and convergence of GMRES for the preconditioned system}\label{tauprecdef2d}
		
    In this section, we define our $\tau$-preconditioner for the linear systems in \eqref{toeplitz-likesystem}, discuss a fast implementation for inversion of the $\tau$-preconditioner  and show that GMRES solver for the preconditioned system has a linear convergence rate independent of $\Delta t$, $\Delta x$ and $\Delta y$.
    
    Clearly, the $N$ many linear systems in \eqref{toeplitz-likesystem} have the common coefficient matrix ${\bf A}$. Hence, for ease of statement, we use the following single linear system to represent the $N$ many linear systems given in \eqref{toeplitz-likesystem}.
    \begin{equation}\label{commontoeplitz-likesystem}
    {\bf A}{\bf u}={\bf b},
    \end{equation}
    with ${\bf u}$ and ${\bf b}$ denoting ${\bf u}^{n}$ and ${\bf b}^n$ for some $n$.
    
    Before defining the preconditioner, we firstly introduce some notations that will be used in later investigation.
    
    For a symmetric Toeplitz matrix ${\bf T}_m\in\mathbb{R}^{m\times m}$ with $[t_0,t_1,...,t_{m-1}]^{\rm T}\in\mathbb{R}^{m\times 1}$ being the first column, define its $\tau$-matrix approximation as \cite{bini1990new}
    \begin{equation}\label{tauopdef}
    	\tau({\bf T}_m):={\bf T}_m-{\bf H}_m,
    \end{equation}
    where ${\bf H}_m$ is a Hankel matrix\footnote{A Hankel matrix is a matrix with entries being constants along the antidiagonals} whose first column and last column are $[t_2,t_3,...,t_{m-1},0,0]^{\rm T}$ and $[0,0,t_{m-1},...,t_3,t_2]^{\rm T}$  are the first column and the last row, respectively. An interesting property of the $\tau$-matrix defined in \eqref{tauopdef} approximation is that it is diagonalizable by the sine transform matrix; i.e.,
    \begin{equation}\label{taumatdiag}
    	\tau({\bf T}_m)={\bf Q}_m{\bf \Lambda}_m{\bf Q}_m^{\rm T},
    \end{equation}
    where ${\bf \Lambda}_m={\rm diag}(\lambda_i)_{i=1}^{m}$ is a diagonal matrix with
    \begin{equation}\label{lambdaicomp}
    	\lambda_{i}=t_0+2\sum\limits_{j=1}^{m-1}t_j\cos\left(\frac{\pi ij}{m+1}\right),\quad  i\in 1\wedge m,
    \end{equation}
    \begin{equation}\label{sinematdef}
    	{\bf Q}_{m}:=\sqrt{\frac{2}{m+1}}\left[\sin\left(\frac{\pi jk}{m+1}\right)\right]_{j,k=1}^{m}
    \end{equation}
    is a sine transform matrix. It is well-known that $Q_m$ is real symmetric and orthogonal for any $m\geq 1$. The product between matrix ${\bf Q}_m$ and a given vector of length $m$ can be fast computed within $\mathcal{O}(m\log m)$ operations using FST.
    Moreover, we note that the diagonal entries $\{\lambda_i\}_{i=1}^{m}$ defined in \eqref{lambdaicomp} can be fast computed within $\mathcal{O}(m\log m)$ operations using fast cosine transform (FCT).

    We then define the $\tau$-preconditioner for \eqref{commontoeplitz-likesystem} as follows
    \begin{equation}\label{pdef}
    {\bf P}={\bf I}_J+\eta_x\bar{\bf A}_x+\eta_y\bar{\bf A}_y,
    \end{equation}
    where
    \begin{align*}
    &\bar{\bf A}_x=\bar{d}{\bf I}_{M_y}\otimes\tau({\bf S}_{\alpha,M_x}), \quad \bar{\bf A}_y=\bar{e}\tau({\bf S}_{\beta,M_y})\otimes{\bf I}_{M_x},\qquad\bar{d}=\sqrt{\check{d}\hat{d}},\quad \bar{e}=\sqrt{\check{e}\hat{e}}.
    \end{align*}
   Instead of solving \eqref{commontoeplitz-likesystem}, we employ GMRES solver to solve the following equivalent preconditioned system
   \begin{equation}\label{precsystem}
   	{\bf P}^{-1}{\bf A}{\bf u}={\bf P}^{-1}{\bf b}.
   \end{equation}
   
   In the each GMRES iteration, it requires to compute some matrix-vector product ${\bf P}^{-1}{\bf A}{\bf v}$ for some given vector ${\bf v}$. One can compute ${\bf A}{\bf v}$ first and then compute ${\bf P}^{-1}(\tilde{\bf y})$ with $\tilde{\bf y}={\bf A}{\bf v}$ given. The matrix-vector ${\bf A}{\bf v}$ can be fast computed by utilizing the Toeplitz-like structure of ${\bf A}$ with $\mathcal{O}(J\log J)$ operations; see, e.g., \cite{wang2012fast,ng2004iterative,chan2007introduction}. It thus remains to discuss the computation of ${\bf P}^{-1}(\tilde{\bf y})$ with a given vector $\tilde{\bf y}$. By \eqref{taumatdiag}-\eqref{lambdaicomp} and the properties of Kronecker product, we see that ${\bf P}$ is diagonalizable by a two-dimension sine transform matrix as follows
   \begin{equation}\label{pdiagform}
   	{\bf P}={\bf Q}{\bf \Lambda}{\bf Q}^{\rm T},
   \end{equation}
   where ${\bf Q}={\bf Q}_{M_y}\otimes {\bf Q}_{M_x}$ is a two-dimension sine transform matrix,
   \begin{align*}
   	{\bf \Lambda}={\bf I}+\bar{d}\eta_x{\bf I}_{M_y}\otimes {\bf \Lambda}_{x}+\bar{e}\eta_y{\bf \Lambda}_{y}\otimes {\bf I}_{M_x},
   \end{align*}
   is a diagonal matrix with
   \begin{align}
   	&{\bf \Lambda}_{x}={\rm diag}(\lambda_{x,i})_{i=1}^{M_x},\quad \lambda_{x,i}=s_0^{(\alpha)}+2\sum\limits_{j=1}^{M_x-1}s_j^{(\alpha)}\cos\left(\frac{\pi ij}{M_x+1}\right),\quad i\in 1\wedge M_x,\label{lambdaxdef}\\
   	&{\bf \Lambda}_{y}={\rm diag}(\lambda_{y,i})_{i=1}^{M_y},\quad \lambda_{y,i}=s_0^{(\beta)}+2\sum\limits_{j=1}^{M_y-1}s_j^{(\beta)}\cos\left(\frac{\pi ij}{M_y+1}\right),\quad i\in 1\wedge M_y.\label{lambdaydef}
   \end{align}
   Hence,  ${\bf P}^{-1}(\tilde{\bf y})$ can be computed in the way ${\bf Q}({\bf\Lambda}^{-1}({\bf Q}^{\rm T}\tilde{\bf y}))$, which can be fast implemented within $\mathcal{O}(J\log J)$ operations by applying the FSTs. 
	
   \subsection{Convergence of GMRES solver for the preconditioned system \eqref{precsystem}}
	
   For any Hermitian positive semi-definite matrix ${\bf B}\in\mathbb{R}^{k\times k}$, define $${\bf B}^{z}:={\bf V}^{\rm T}\diag([{\bf D}(i,i)]^z)_{i=1}^{k}{\bf V},\quad z\in\mathbb{R},$$ where ${\bf B}={\bf V}^{\rm T}{\bf D}{\bf V}$ denotes the orthogonal diagonalization of ${\bf B}$. If ${\bf B}$ is Hermitian positive definite (HPD) and $z\in\mathbb{R}$, we rewrite $({\bf B}^{-1})^{z}$ as ${\bf B}^{-z}$  for simplification of the notation.
   
   For any Hermitian matrix ${\bf H}$, denote by $\lambda_{\min}({\bf H})$ and  $\lambda_{\max}({\bf H})$, the minimal eigenvalue and the maximal eigenvalue of ${\bf H}$, respectively.
   
      	For any symmetric matrices ${\bf H}_1,{\bf H}_2\in\mathbb{R}^{m\times m}$, denote ${\bf H}_2 \succ ({\rm or} \succeq) \ {\bf H}_1$
   if ${\bf H}_2-{\bf H}_1$ is positive definite (or semi-definite). Especially, we denote ${\bf H}_2 \succ ({\rm or} \succeq) \ {\bf O}$, if ${\bf H}_2$ itself is positive definite (or semi-definite).
   Also, ${\bf H}_1 \prec ({\rm or} \preceq) \ {\bf H}_2$ and ${\bf O} \prec ({\rm or} \preceq) \ {\bf H}_2$  have the same meanings as those of ${\bf H}_2 \succ ({\rm or} \succeq) \ {\bf H}_1$ and ${\bf H}_2 \succ ({\rm or} \succeq) \ {\bf O}$, respectively.
   
   \begin{lemma}\label{tausgammamhpdlm}
   For any $\gamma\in(1,2)$ and any $m\in\mathbb{N}^{+}$, it holds that $\tau({\bf S}_{\gamma,m})\succ {\bf O}$.
   \end{lemma}
   \begin{proof}
   	By \eqref{lambdaicomp}, we see that the eigenvalues of $\tau({\bf S}_{\gamma,m})$ can be expressed as
   	\begin{align*}
   		\lambda_{i}^{(\gamma,m)}=s_0^{(\gamma)}+2\sum\limits_{j=1}^{m-1}s_j^{(\gamma)}\cos\left(\frac{\pi ij}{m+1}\right),\quad i\in 1\wedge m.
   	\end{align*}
   	 From Property \ref{skprop}${\bf (ii)}-{\bf (iii)}$, we see that
   	\begin{align}
   		\lambda_{i}^{(\gamma,m)}&\geq s_0^{(\gamma)}-2\sum\limits_{j=1}^{m-1}\left|s_j^{(\gamma)}\cos\left(\frac{\pi ij}{m+1}\right)\right|\notag\\
   		&\geq s_0^{(\gamma)}-2\sum\limits_{j=1}^{m-1}\left|s_j^{(\gamma)}\right|\notag\\
   		&=s_0^{(\gamma)}+2\sum\limits_{j=1}^{m-1}s_j^{(\gamma)}>0,\quad  i\in 1\wedge m.
   	\end{align}
   	Moreover, $\tau({\bf S}_{\gamma,m})$ is clearly Hermitian. Hence, $\tau({\bf S}_{\gamma,m})\succ {\bf O}$. The proof is complete.
   \end{proof}
   
   With Lemma \ref{tausgammamhpdlm}, we immediately have
   \begin{equation}\label{phpdeq}
   	{\bf P}={\bf I}_J+\eta_x\bar{d}{\bf I}_{M_y}\otimes \tau({\bf S}_{\alpha,M_x})+\eta_y\bar{e} \tau({\bf S}_{\beta,M_y})\otimes{\bf I}_{M_x}\succeq {\bf I}.
   \end{equation}

   \eqref{phpdeq} implies that ${\bf P}^{\frac{1}{2}}$ and ${\bf P}^{-\frac{1}{2}}$ exist.
    To analyze the convergence rate of GMRES for the preconditioned system \eqref{precsystem}, we firstly investigate the convergence behavior of GMRES solver for the following auxiliary system
    \begin{equation}\label{auxsystem}
    {\bf P}^{-\frac{1}{2}}{\bf A}{\bf P}^{-\frac{1}{2}}\hat{\bf u}={\bf P}^{-\frac{1}{2}}{\bf b},
    \end{equation}
    where the solution of \eqref{auxsystem} and the solution of \eqref{precsystem} are related by ${\bf u}={\bf P}^{-\frac{1}{2}}\hat{\bf u}$. The reason is explained in what follows.
    
    The convergence behavior of GMRES is closely related to the Krylov subspace. For a square matrix ${\bf Z}\in \mathbb{R}^{m \times m}$ and a vector ${\bf x}\in \mathbb{R}^{m\times 1}$, a Krylov subspace of degree $j\geq 1$ is defined as follows
    $$\mathcal{K}_j({\bf Z},{\bf x}):=\text{span}\{{\bf x},{\bf Z}{\bf x},{\bf Z}^2{\bf x},\dots,{\bf Z}^{j-1}{\bf x} \}.$$
    For a set $\mathcal{S}$ and a point $z$, we denote
    $$z+\mathcal{S}:=\{z+x|x\in \mathcal{S}\}.$$
    We recall the relation between the iterative solution by GMRES and the Krylov subspace in the following lemma.
    \begin{lemma}\label{krylov}\textnormal{(see, e.g., \cite{saad2003})}
    	For a non-singular $m\times m$ real linear system ${\bf Z} {\bf v} = {\bf w}$, let ${\bf v}_j$ be the iterative solution by GMRES at $j$-th ($j\geq1$) iteration step with ${\bf v}_0$ as initial guess. Then, the $j$-th iteration solution ${\bf v}_j$ minimize the residual error over the Krylov subspace $\mathcal{K}_j({\bf Z},{\bf r}_0)$ with ${\bf r}_0= {\bf w}-{\bf Z}{\bf v}_0,$ i.e.,	
    	$$
    	{\bf v}_{j}=\underset{\mathbf{y} \in \mathbf{v}_{0}+\mathcal{K}_{j}\left(\mathbf{Z}, \mathbf{r}_{0}\right)}{\arg \min }\|\mathbf{w}-\mathbf{Z} \mathbf{y}\|_{2} .
    	$$
    \end{lemma}

    \begin{lemma}\label{residual}
    	Let $\hat{\mathbf{u}}_{0}$ be the initial guess for \eqref{auxsystem}. Let $\mathbf{u}_{0}:=\mathbf{P}^{-\frac{1}{2}} \hat{\mathbf{u}}_{0}$ be the initial guess for \eqref{precsystem}. Let $\mathbf{u}_{j}\left(\hat{\mathbf{u}}_{j}\right.$, respectively $)$ be the jth $(j \geq 1)$ iteration solution derived by applying GMRES solver to \eqref{precsystem} $($\eqref{auxsystem}, respectively $)$ with $\mathbf{u}_{0}$ $($ $\hat{\mathbf{u}}_{0}$, respectively $)$ as initial guess. Then,
    	$$
    	\left\|\mathbf{r}_{j}\right\|_{2} \leq \left\|\hat{\mathbf{r}}_{j}\right\|_{2},\quad j=0,1,2,...
    	$$
    	where $\mathbf{r}_{j}:=\mathbf{P}^{-1} \mathbf{f}-\mathbf{P}^{-1} \mathbf{A} \mathbf{u}_{j}$ $($ $\hat{\mathbf{r}}_{j}:=\mathbf{P}^{-\frac{1}{2}} \mathbf{f}-\mathbf{P}^{-\frac{1}{2}} \mathbf{A} \mathbf{P}^{-\frac{1}{2}}\hat{\mathbf{u}}_{j}$, respectively $)$ denotes the residual vector at $j$-th GMRES iteration for \eqref{precsystem} $($ \eqref{auxsystem}, respectively $)$.
    \end{lemma}

    \begin{proof}
    	By applying Lemma \ref{krylov} to \eqref{auxsystem}, we see that
    	$$
    	\hat{\mathbf{u}}_{j}-\hat{\mathbf{u}}_{0} \in \mathcal{K}_{j}\left(\mathbf{P}^{-\frac{1}{2}} \mathbf{A} {\bf P}^{-\frac{1}{2}}, \hat{\mathbf{r}}_{0}\right),
    	$$
    	where $\hat{\mathbf{r}}_{0}=\mathbf{P}^{-\frac{1}{2}} \mathbf{f}-\mathbf{P}^{-\frac{1}{2}} \mathbf{A} {\bf P}^{-\frac{1}{2}} \hat{\mathbf{u}}_{0}$. Notice that $\left(\mathbf{P}^{-\frac{1}{2}} \mathbf{A} {\bf P}^{-\frac{1}{2}}\right)^{k}=\mathbf{P}^{\frac{1}{2}}\left(\mathbf{P}^{-1} \mathbf{A}\right)^{k} \mathbf{P}^{-\frac{1}{2}}$ for each $k \geq 0$. Therefore,
    	$$
    	\begin{aligned}
    		\mathcal{K}_{j}\left(\mathbf{P}^{-\frac{1}{2}} \mathbf{A} \mathbf{P}^{-\frac{1}{2}}, \hat{\mathbf{r}}_{0}\right) &=\operatorname{span}\left\{\left(\mathbf{P}^{-\frac{1}{2}} \mathbf{A} \mathbf{P}^{-\frac{1}{2}}\right)^{k}\left(\mathbf{P}^{-\frac{1}{2}} \mathbf{f}-\mathbf{P}^{-\frac{1}{2}} \mathbf{A} \mathbf{P}^{-\frac{1}{2}} \hat{\mathbf{u}}_{0}\right)\right\}_{k=0}^{j-1} \\
    		&=\operatorname{span}\left\{\mathbf{P}^{\frac{1}{2}}\left(\mathbf{P}^{-1} \mathbf{A}\right)^{k} \mathbf{P}^{-\frac{1}{2}}\left(\mathbf{P}^{-\frac{1}{2}} \mathbf{f}-\mathbf{P}^{-\frac{1}{2}} \mathbf{A} {\bf P}^{-\frac{1}{2}} \hat{\mathbf{u}}_{0}\right)\right\}_{k=0}^{j-1} \\
    		&=\operatorname{span}\left\{\mathbf{P}^{\frac{1}{2}}\left(\mathbf{P}^{-1} \mathbf{A}\right)^{k}\left(\mathbf{P}^{-1} \mathbf{f}-\mathbf{P}^{-1} \mathbf{A} \mathbf{u}_{0}\right)\right\}_{k=0}^{j-1}
    	\end{aligned}
    	$$
    	where the last equality comes from the fact that $\mathbf{u}_{0}=\mathbf{P}^{-\frac{1}{2}} \hat{\mathbf{u}}_{0}$. That means
    	\begin{equation*}
    		\hat{\mathbf{u}}_{j}-\hat{\mathbf{u}}_{0}\in\operatorname{span}\left\{\mathbf{P}^{\frac{1}{2}}\left(\mathbf{P}^{-1} \mathbf{A}\right)^{k}\left(\mathbf{P}^{-1} \mathbf{f}-\mathbf{P}^{-1} \mathbf{A} \mathbf{u}_{0}\right)\right\}_{k=0}^{j-1}
    	\end{equation*}
    	Therefore,
    	$$
    	\mathbf{P}^{-\frac{1}{2}} \hat{\mathbf{u}}_{j}-\mathbf{u}_{0}=\mathbf{P}^{-\frac{1}{2}}\left(\hat{\mathbf{u}}_{j}-\hat{\mathbf{u}}_{0}\right) \in \operatorname{span}\left\{\left(\mathbf{P}^{-1} \mathbf{A}\right)^{k}\left(\mathbf{P}^{-1} \mathbf{f}-\mathbf{P}^{-1} \mathbf{A} \mathbf{u}_{0}\right)\right\}_{k=0}^{j-1}=\mathcal{K}_{j}\left(\mathbf{P}^{-1} \mathbf{A}, \mathbf{r}_{0}\right),
    	$$
    	where $\mathbf{r}_{0}=\mathbf{P}^{-1} \mathbf{f}-\mathbf{P}^{-1} \mathbf{A u}_{0}$. In other words,
    	$$
    	\mathbf{P}^{-\frac{1}{2}} \hat{\mathbf{u}}_{j} \in \mathbf{u}_{0}+\mathcal{K}_{j}\left(\mathbf{P}^{-1} \mathbf{A}, \mathbf{r}_{0}\right) .
    	$$
    	By applying Lemma \ref{krylov} to \eqref{precsystem}, we know that
    	$$
    	\mathbf{u}_{j}=\underset{\mathbf{y} \in \mathbf{u}_{0}+\mathcal{K}_{j}\left(\mathbf{P}^{-1} \mathbf{A}, \mathbf{r}_{0}\right)}{\arg \min }\left\|\mathbf{P}^{-1} \mathbf{f}-\mathbf{P}^{-1} \mathbf{A} \mathbf{y}\right\|_{2} .
    	$$
    	Therefore,
    
    	\begin{align*}
    		\left\|\mathbf{r}_{j}\right\|_{2}=\left\|\mathbf{P}^{-1} \mathbf{f}-\mathbf{P}^{-1} \mathbf{A} \mathbf{u}_{j}\right\|_{2} & \leq\left\|\mathbf{P}^{-1} \mathbf{f}-\mathbf{P}^{-1} \mathbf{A} \mathbf{P}^{-\frac{1}{2}} \hat{\mathbf{u}}_{j}\right\|_{2} \\
    		&=\left\|\mathbf{P}^{-\frac{1}{2}} \hat{\mathbf{r}}_{j}\right\|_{2} \\
    		&=\sqrt{\hat{\mathbf{r}}_{j}^{\mathrm{T}} \mathbf{P}^{-1} \hat{\mathbf{r}}_{j}} \\
    		&\leq ||\hat{\bf r}_j||_2,
    	\end{align*}
    	where the last inequality comes from \eqref{phpdeq}.
    \end{proof}
    
    Lemma \ref{residual} shows that GMRES solver for \eqref{precsystem} converges always faster than GMRES solver for \eqref{auxsystem} whenever the initial guess ${\bf u}_0$ of \eqref{precsystem} and the initial guess $\hat{\bf u}_0$ of \eqref{auxsystem} are related by ${\bf u}_0:={\bf P}^{-\frac{1}{2}}\hat{\bf u}_0$. Because of this reason, it suffices to investigate the convergence behavior of GMRES solver on the auxiliary system \eqref{auxsystem}.

    For a real square matrix ${\bf Z}$, denote
    \begin{equation*}
    	\mathcal{H}({\bf Z}):=\frac{{\bf Z}+{\bf Z}^{\rm T}}{2},\quad \mathcal{S}({\bf Z}):=\frac{{\bf Z}-{\bf Z}^{\rm T}}{2}.
    \end{equation*}
    For a square matrix ${\bf C}$, denote by $\rho({\bf C})$, the spectral of ${\bf C}$.
  
  \begin{lemma}\textnormal{(see \cite[Proposition 7.3]{elman2014finite})}\label{gmrescvglm}
  	Let ${\bf Z}{\bf v}={\bf w}$ be a real square linear system with $\mathcal{H}({\bf Z})\succ{\bf O}$. Then, the residuals of the iterates generated by applying (restarted or non-restarted) GMRES to solve ${\bf Z}{\bf v}={\bf w}$ satisfy
  	\begin{equation*}
  		||{\bf r}_{k}||_2\leq\left(1-\frac{\lambda_{\min}(\mathcal{H}({\bf Z}))^2}{\lambda_{\min}(\mathcal{H}({\bf Z}))\lambda_{\max}(\mathcal{H}({\bf Z}))+\rho(\mathcal{S}({\bf Z}))^2}\right)^{k/2}||{\bf r}_{0}||_2,
  	\end{equation*}
  	where ${\bf r}_{k}={\bf w}-{\bf Z}{\bf v}_k$ with ${\bf v}_k$ $(k\geq 1)$ being the iterative solution at $k$th GMRES iteration and ${\bf v}_0$ being an arbitrary initial guess.
  \end{lemma}
  
  To apply Lemma \ref{gmrescvglm} to the auxiliary system \eqref{auxsystem}, we need to show ${\bf P}^{-\frac{1}{2}}{\bf A}  {\bf P}^{-\frac{1}{2}}\succ{\bf O}$, to estimate lower and upper bounds of eigenvalues of $\mathcal{H}(  {\bf P}^{-\frac{1}{2}}{\bf A}  {\bf P}^{-\frac{1}{2}})$ and upper bound of spectral radius of $\mathcal{S}(  {\bf P}^{-\frac{1}{2}}{\bf A}  {\bf P}^{-\frac{1}{2}})$.

  For a square matrix ${\bf C}$, denote by $\sigma({\bf C})$, the spectrum of ${\bf C}$.
  Referring to the discussion in \cite{huangxin2022}, one can show the following lemma
  \begin{lemma}\label{taumatprecspectralm}
  	For any $\gamma\in(1,2)$ and any $m\in\mathbb{N}^{+}$, ${\bf O}\prec \frac{1}{2}\tau({\bf S}_{\gamma,m})\prec {\bf S}_{\gamma,m}\prec \frac{3}{2}\tau({\bf S}_{\gamma,m})$.
  \end{lemma}
 \begin{proof}
 	See the Appendix \ref{taumatpreclmproof}.
 \end{proof}

 Denote
  \begin{equation}\label{intermediateprec}
  	\tilde{\bf P}:={\bf I}_J+\eta_x\tilde{\bf A}_x+\eta_y\tilde{\bf A}_y,
  \end{equation}
  with
  \begin{equation*}
  	\tilde{\bf A}_x:=\bar{d}{\bf I}_{M_y}\otimes {\bf S}_{\alpha,M_x},\quad \tilde{\bf A}_y:=\bar{e} {\bf S}_{\beta,M_y}\otimes {\bf I}_{M_x}
  \end{equation*}
  Lemma \ref{tausgammamhpdlm} implies that
  \begin{align}
  {\bf O}\prec	\frac{1}{2}{\bf P}&= \frac{1}{2}[{\bf I}_J+\eta_x\bar{d}{\bf I}_{M_y}\otimes\tau({\bf S}_{\alpha,M_x})+\eta_y\bar{e}\tau({\bf S}_{\beta,M_y})\otimes {\bf I}_{M_x}]\notag\\
  	&\prec \underbrace{ {\bf I}_J+\eta_x\bar{d}{\bf I}_{M_y}\otimes{\bf S}_{\alpha,M_x}+\eta_y\bar{e}{\bf S}_{\beta,M_y}\otimes {\bf I}_{M_x}}_{:=\tilde{\bf P}}\notag\\
  	&\prec {\bf I}_J+(3/2)\eta_x\bar{d}{\bf I}_{M_y}\otimes\tau({\bf S}_{\alpha,M_x})+(3/2)\eta_y\bar{e}\tau({\bf S}_{\beta,M_y})\otimes {\bf I}_{M_x}\prec \frac{3}{2}{\bf P}\label{ptildpcontorleq}.
  \end{align}
  \eqref{ptildpcontorleq} means that ${\bf P}$ is a good approximation to $\tilde{\bf P}$, while $\tilde{\bf P}$ is obtained by replacing ${\bf D}$ and ${\bf E}$ appearing in ${\bf A}$ with their mean values. That explains why we choose ${\bf P}$ as preconditioner of ${\bf A}$.
  
 Let $\lambda_{*}$ be an eigenvalue of $\mathcal{S}({\bf P}^{-\frac{1}{2}}{\bf A}{\bf P}^{-\frac{1}{2}})$ with $|\lambda_{*}|=\rho(\mathcal{S}({\bf P}^{-\frac{1}{2}}{\bf A}{\bf P}^{-\frac{1}{2}}))$. Clearly, $\mathcal{S}({\bf P}^{-\frac{1}{2}}{\bf A}{\bf P}^{-\frac{1}{2}})={\bf P}^{-\frac{1}{2}}\mathcal{S}({\bf A}){\bf P}^{-\frac{1}{2}}$. By matrix similarity, we have $\sigma({\bf P}^{-\frac{1}{2}}\mathcal{S}({\bf A}){\bf P}^{-\frac{1}{2}})=\sigma({\bf P}^{-1}\mathcal{S}({\bf A}))$. Therefore, $\lambda_{*}$ is also an eigenvalue of ${\bf P}^{-1}\mathcal{S}({\bf A})$. Let ${\bf z}_{0}$ be the eigenvector of ${\bf P}^{-1}\mathcal{S}({\bf A})$ corresponding to the eigenvalue $\lambda_{*}$. We then have
 \begin{equation*}
\mathcal{S}({\bf A}){\bf z}_{0}=\lambda_{*}{\bf P}{\bf z}_{0}\Longrightarrow \lambda_{*}=\frac{{\bf z}_{0}^{*}\mathcal{S}({\bf A}){\bf z}_{0}}{{\bf z}_{0}^{*}{\bf P}{\bf z}_{0}}.
 \end{equation*}
 With \eqref{ptildpcontorleq}, we have
 \begin{align}
\rho(\mathcal{S}({\bf P}^{-\frac{1}{2}}{\bf A}{\bf P}^{-\frac{1}{2}}))=|\lambda_{*}|=\frac{|{\bf z}_{0}^{*}\mathcal{S}({\bf A}){\bf z}_{0}|}{{\bf z}_{0}^{*}{\bf P}{\bf z}_{0}}&=\frac{|{\bf z}_{0}^{*}\mathcal{S}({\bf A}){\bf z}_{0}|}{{\bf z}_{0}^{*}\tilde{\bf P}{\bf z}_{0}}\times \frac{{\bf z}_{0}^{*}\tilde{\bf P}{\bf z}_{0}}{{\bf z}_{0}^{*}{\bf P}{\bf z}_{0}}\notag\\
&\leq \frac{3}{2}\max \limits_{0\neq {\bf z}\in\mathbb{C}^{J\times J}}\frac{|{\bf z}^{*}\mathcal{S}({\bf A}){\bf z}|}{{\bf z}^{*}\tilde{\bf P}{\bf z}}.\label{skewpartcontrl1}
 \end{align}
 Similarly, one can show that
 \begin{align}
 &	\lambda_{\max}(\mathcal{H}({\bf P}^{-\frac{1}{2}}{\bf A}{\bf P}^{-\frac{1}{2}}))\leq \frac{3}{2} \max \limits_{0\neq {\bf z}\in\mathbb{R}^{J\times J}}\frac{{\bf z}^{\rm T}\mathcal{H}({\bf A}){\bf z}}{{\bf z}^{\rm T}\tilde{\bf P}{\bf z}},\label{hpartubcontrol1}\\
 &		\lambda_{\min}(\mathcal{H}({\bf P}^{-\frac{1}{2}}{\bf A}{\bf P}^{-\frac{1}{2}}))\geq \frac{1}{2} \min \limits_{0\neq {\bf z}\in\mathbb{R}^{J\times J}}\frac{{\bf z}^{\rm T}\mathcal{H}({\bf A}){\bf z}}{{\bf z}^{\rm T}\tilde{\bf P}{\bf z}}.\label{hpartlbcontrol1}
 \end{align}
 \eqref{skewpartcontrl1}--\eqref{hpartlbcontrol1} show that to estimate upper bounds of  $\lambda_{\max}(\mathcal{H}( {\bf P}^{-\frac{1}{2}}{\bf A}{\bf P}^{-\frac{1}{2}}))$ and $ \rho(\mathcal{S}({\bf P}^{-\frac{1}{2}}{\bf A}{\bf P}^{-\frac{1}{2}}))$ and a  lower bound of  $\lambda_{\min}(\mathcal{H}( {\bf P}^{-\frac{1}{2}}{\bf A}{\bf P}^{-\frac{1}{2}}))$, it suffices to estimate upper bounds of $\max \limits_{0\neq {\bf z}\in\mathbb{R}^{J\times J}}\frac{{\bf z}^{\rm T}\mathcal{H}({\bf A}){\bf z}}{{\bf z}^{\rm T}\tilde{\bf P}{\bf z}}$ and $\max \limits_{0\neq {\bf z}\in\mathbb{C}^{J\times J}}\frac{|{\bf z}^{*}\mathcal{S}({\bf A}){\bf z}|}{{\bf z}^{*}\tilde{\bf P}{\bf z}}$ and to estimate a lower bound of $\min\limits_{0\neq {\bf z}\in\mathbb{R}^{J\times J}}\frac{{\bf z}^{\rm T}\mathcal{H}({\bf A}){\bf z}}{{\bf z}^{\rm T}\tilde{\bf P}{\bf z}}$.

  Property \ref{skprop}${\bf (ii)}$-${\bf (iii)}$ together with the well-known Gershgorin’s Theorem immediately imply that  ${\bf S}_{\gamma,M}\succ {\bf O}$ for any $M\geq 1$ and $\gamma\in(1,2)$, i.e.,
 \begin{equation}\label{sgamhpd}
{\bf S}_{\gamma,M}\succ {\bf O},\quad \forall\gamma\in(1,2),\quad \forall M\geq 1.
 \end{equation}
  
  	For a diagonal matrix ${\bf Z}=\diag(z_1,z_2,...,z_m)\in\mathbb{R}^{m\times m}$, denote
  \begin{equation*}
  \max({\bf Z})=\max\limits_{1\leq i\leq m}z_i,\quad	\min({\bf Z})=\min\limits_{1\leq i\leq m}z_i, \quad\nabla({\bf Z}):=\max\limits_{1\leq i,j\leq m,~i\neq j}\frac{|z_i-z_j|}{|i-j|}.
  \end{equation*}
  For a symmetric matrix ${\bf H}\in\mathbb{R}^{m\times m}$ and a non-negative diagonal matrix ${\bf Z}\in\mathbb{R}^{m\times m}$, denote
  \begin{equation*}
  	\Delta_{\bf H}({\bf Z}):={\bf Z}{\bf H}+{\bf H}{\bf Z}-2{\bf Z}^{\frac{1}{2}}{\bf H}{\bf Z}^{\frac{1}{2}}.
  \end{equation*}
  \begin{lemma}\label{resmatlemm}
  	Let ${\bf Z}=\diag(z_1,z_2,...,z_M)\in\mathbb{R}^{M\times M}$. Let $\tilde{b}>0$ be  such that 
  	\begin{equation*}
    \nabla({\bf Z})\leq \frac{\tilde{b}}{M+1}.
  	\end{equation*}
  	Assume that there exists  $\check{b}>0$ such that $\min({\bf Z})\geq \check{b}$.
  	Then, 
  	\begin{equation*}
  	||\Delta_{{\bf S}_{\gamma,M}}({\bf Z})||_2\leq\frac{\mu_{\gamma}(||\{s_{k}^{({\gamma})}\}||_{\mathcal{D}_{{\gamma}}},\tilde{b},\check{b}) }{(M+1)^{\gamma}},
  	\end{equation*}
  	where ${\bf S}_{\gamma,M}$ is defined in \eqref{sgamamdef},
  	\begin{equation*}
  	\mu_{\gamma}(x,y,z):=\frac{xy^2}{2z(2-\gamma)},\quad x,y,z>0,\quad\gamma\in(1,2).
  	\end{equation*}
  \end{lemma}
  \begin{proof}
  See the Appendix \ref{retmatlmproof}.
  \end{proof}

    	\begin{lemma}\label{hpartcontrollm}
    	Let ${\bf Z}=\diago(z_1,z_2,...,z_M)\in\mathbb{R}^{M\times M}.$  Let $\tilde{b}>0$ be  such that 
    	\begin{equation*}
    		\nabla({\bf Z})\leq \frac{\tilde{b}}{M+1}.
    	\end{equation*} 
    	Assume there exits $\hat{b}\geq \check{b}>0$ such that
    	\begin{equation*}
    	\hat{b}\geq \max({\bf Z})\geq \min({\bf Z})\geq \check{b}.
    	\end{equation*}
    	Then, for any $\theta_1\in[0,1)$ and any $\theta_2\in(0,+\infty)$, it holds that
    	\begin{align*}
    	2\theta_1\check{b}{\bf S}_{\gamma,M}-\left(\frac{\mu_{\gamma}(||\{s_{k}^{({\gamma})}\}||_{\mathcal{D}_{{\gamma}}},\tilde{b},(1-\theta_1)\check{b})}{(1+M)^{\gamma}}\right){\bf I}_M&\prec {\bf Z}{\bf S}_{\gamma,M}+{\bf S}_{\gamma,M}{\bf Z}\\
    	&\prec 2(\hat{b}+\theta_2){\bf S}_{\gamma,M}+\left(\frac{\mu_{\gamma}(||\{s_{k}^{({\gamma})}\}||_{\mathcal{D}_{{\gamma}}},\tilde{b},\theta_2)}{(1+M)^{\gamma}}\right){\bf I}_{M}.
    	\end{align*}
    	where the function $\mu_{\gamma}(\cdot,\cdot,\cdot)$ is defined by Lemma \ref{resmatlemm}.
    \end{lemma}
    \begin{proof}
    See the Appendix \ref{hpartcontrllmproof}.
    \end{proof}

    In \eqref{vxydef}, a $x$-dominant ordering of the spatial grid points is defined. Similarly, one can define a $y$-dominant ordering of the spatial grid points as follows
    \begin{equation*}
    	{\bf V}_{y,x}=(G_{1,1},G_{1,2},...,G_{1,M_y},G_{2,1},G_{2,2},...,G_{2,M_y},......,G_{M_x,1},G_{M_x,2},...,G_{M_x,M_y})^{\rm T}.
    \end{equation*}
        There exists a permutation matrix ${\bf P}_{x\leftrightarrow y}\in\mathbb{R}^{J\times J}$ such that
    \begin{equation*}
    	{\bf V}_{y,x}={\bf P}_{x\leftrightarrow y}{\bf V}_{x,y}.
    \end{equation*}
    
    Define the set of functions that are Lipschitz continuous with respect to the  first and the second variable , respectively as follows
    \begin{align*}
    &\mathcal{L}_1(\Omega):=\left\{w:\Omega\rightarrow\mathbb{R}\ \bigg| |w|_{\mathcal{L}_1(\mathcal{S})}:=\sup\limits_{y\in(l_2,r_2)} \ \sup\limits_{x,\tilde{x}\in (l_1,r_1),~x\neq \tilde{x}}\frac{|w(x,y)-w(\tilde{x},y)|}{|x-\tilde{x}|}<+\infty\right\}\\
    &\mathcal{L}_2(\Omega):=\left\{w:\Omega\rightarrow\mathbb{R}\ \bigg| |w|_{\mathcal{L}_2(\mathcal{S})}:=\sup\limits_{x\in(l_1,r_1)} \ \sup\limits_{y,\tilde{y}\in (l_2,r_2),~y\neq \tilde{y}}\frac{|w(x,y)-w(x,\tilde{y})|}{|y-\tilde{y}|}<+\infty\right\}.
    \end{align*}

   	\begin{lemma}\label{2dhpartcontrollm}
   Assume $d\in\mathcal{L}_{1}(\Omega)$, $e\in\mathcal{L}_{2}(\Omega)$.  
   	Then, for any $\theta_1,\theta_2\in[0,1)$ and any $\theta_3,\theta_4\in(0,+\infty)$, it holds that
   	\begin{align*}
   		&2\theta_1\check{d}({\bf I}_{M_y}\otimes {\bf S}_{\alpha,M_x})-s_{1,1}(\theta_1)\Delta x^{\alpha}{\bf I}_{J}\prec {\bf A}_x+{\bf A}_x^{\rm T}\prec 2(\hat{d}+\theta_3)({\bf I}_{M_y}\otimes {\bf S}_{\alpha,M_x})+s_{1,2}(\theta_3)\Delta x^{\alpha}{\bf I}_{J},\\
   		&2\theta_2\check{e}( {\bf S}_{\beta,M_y}\otimes{\bf I}_{M_x})-s_{2,1}(\theta_2)\Delta y^{\beta}{\bf I}_{J}\prec {\bf A}_y+{\bf A}_y^{\rm T}\prec 2(\hat{e}+\theta_4)( {\bf S}_{\beta,M_y}\otimes{\bf I}_{M_x})+s_{2,2}(\theta_4)\Delta y^{\beta}{\bf I}_{J},
   	\end{align*}
	where $s_{1,1},s_{1,2},s_{2,1},s_{2,2}$ are positive constants independent of discretization step-sizes given as follows
\begin{align*}
	&s_{1,1}(\theta_1):=\frac{||\{s_k^{(\alpha)}\}||_{\mathcal{D}_{\alpha}}|d|^2_{\mathcal{L}_1(\Omega)}}{2(1-\theta_1)\check{d}(2-\alpha)(r_1-l_1)^{\alpha-2}},\quad 	s_{1,2}(\theta_3):=\frac{||\{s_k^{(\alpha)}\}||_{\mathcal{D}_{\alpha}}|d|^2_{\mathcal{L}_1(\Omega)}}{2\theta_3(2-\alpha)(r_1-l_1)^{\alpha-2}},\\
	&s_{2,1}(\theta_2):=\frac{||\{s_k^{(\beta)}\}||_{\mathcal{D}_{\beta}}|e|^2_{\mathcal{L}_2(\Omega)}}{2(1-\theta_2)\check{e}(2-\beta)(r_2-l_2)^{\beta-2}},\quad 	s_{2,2}(\theta_4):=\frac{||\{s_k^{(\beta)}\}||_{\mathcal{D}_{\beta}}|e|^2_{\mathcal{L}_2(\Omega)}}{2\theta_4(2-\beta)(r_2-l_2)^{\beta-2}}
\end{align*}
   \end{lemma}
   \begin{proof}
   	See the Appendix \ref{d2hpartcontrollmproof}.
   \end{proof}

   \begin{lemma}\label{prehpartminmaxlm}
   	 Assume $d\in\mathcal{L}_{1}(\Omega)$, $e\in\mathcal{L}_{2}(\Omega)$.  Take $\Delta t\leq c_0$ with the constant $c_0$ independent of step-sizes defined as follows
   	\begin{equation*}
   	c_0=\min\left\{\frac{1-c_1}{b_1},\frac{c_2-1}{b_2}\right\}>0, 
   	\end{equation*}
   where
   	\begin{align*}
   	&c_1=\min\left\{\sqrt{\frac{\check{d}}{2\hat{d}}},\sqrt{\frac{\check{e}}{2\hat{e}}}\right\}\in(0,1),\quad c_2=\max\left\{\sqrt{\frac{2\hat{d}}{\check{d}}},\sqrt{\frac{2\hat{e}}{\check{e}}}\right\}>\sqrt{2}\\
   	 &b_1=\frac{||\{s_k^{(\alpha)}\}_{k\geq 0}||_{\mathcal{D}_{\alpha}}|d|^2_{\mathcal{L}_1(\Omega)}(r_1-l_1)^{2}}{4(1-\sqrt{2}/2)\check{d}(2-\alpha)}+\frac{||\{s_k^{(\beta)}\}_{k\geq 0}||_{\mathcal{D}_{\beta}}|e|^2_{\mathcal{L}_2(\Omega)}(r_2-l_2)^{2}}{4(1-\sqrt{2}/2)\check{e}(2-\beta)}>0,\\
   	 &b_2=\frac{||\{s_k^{(\alpha)}\}_{k\geq 0}||_{\mathcal{D}_{\alpha}}|d|^2_{\mathcal{L}_1(\Omega)}(r_1-l_1)^{2}}{4(\sqrt{2}-1)\hat{d}(2-\alpha)}+\frac{||\{s_k^{(\beta)}\}_{k\geq 0}||_{\mathcal{D}_{\beta}}|e|^2_{\mathcal{L}_2(\Omega)}(r_2-l_2)^{2}}{4(\sqrt{2}-1)\hat{e}(2-\beta)}>0.
   	\end{align*}
   	Then, for any non-zero vector ${\bf z}\in\mathbb{R}^{J\times J}$, it holds that
   	\begin{equation*}
   	0<c_1\leq
    \frac{{\bf z}^{\rm T}\mathcal{H}({\bf A}){\bf z}}{{\bf z}^{\rm T}\tilde{\bf P}{\bf z}} \leq c_2.
   	\end{equation*}
   \end{lemma}
   \begin{proof}
   	Take $\theta_1=\theta_2=\frac{1}{\sqrt{2}}$. By the assumption, $\Delta t\leq c_0\leq \frac{1-c_1}{b_1}$. That means
   	\begin{equation*}
   	c_1\leq 1-\Delta t b_1
   	\end{equation*}
   	 By Lemma \ref{2dhpartcontrollm}, we have
   	\begin{align}
   \mathcal{H}({\bf A})	=&{\bf I}_J+\eta_x\frac{{\bf A}_x+{\bf A}_x^{\rm T}}{2}+\eta_y\frac{{\bf A}_y+{\bf A}_y^{\rm T}}{2}\notag\\
   \succ &[1-s_{1,1}(\theta_1)\eta_x\Delta x^{\alpha}/2-s_{2,1}(\theta_2)\eta_y\Delta y^{\beta}/2]{\bf I}_J\notag\\
   &+\theta_1\check{d}\eta_x({\bf I}_{M_y}\otimes {\bf S}_{\alpha,M_x})+\theta_2\check{e}\eta_y({\bf S}_{\beta,M_y}\otimes {\bf I}_{M_x})\notag\\
   =&(1-\Delta tb_1){\bf I}_J+\sqrt{\frac{\check{d}}{2\hat{d}}}\bar{d}\eta_x({\bf I}_{M_y}\otimes {\bf S}_{\alpha,M_x})+\sqrt{\frac{\check{e}}{2\hat{e}}}\bar{e}\eta_y({\bf S}_{\beta,M_y}\otimes {\bf I}_{M_x})\notag\\
   \succeq &c_1{\bf I}_J+c_1\bar{d}\eta_x({\bf I}_{M_y}\otimes {\bf S}_{\alpha,M_x})+c_1\bar{e}\eta_y({\bf S}_{\beta,M_y}\otimes {\bf I}_{M_x})=c_1\tilde{\bf P}.\label{halbcontrlbytildp}
   	\end{align}
   	
   	Take  $\theta_3=(\sqrt{2}-1)\hat{d}$, $\theta_4=(\sqrt{e}-1)\hat{e}$. By the assumption, $\Delta t\leq c_0\leq \frac{c_2-1}{b_2}$. That means
   	\begin{equation*}
   	1+\Delta t b_2\leq c_2.
   	\end{equation*}
   	By Lemma \ref{2dhpartcontrollm}, we have
   	\begin{align*}
      \mathcal{H}({\bf A})	=&{\bf I}_J+\eta_x\frac{{\bf A}_x+{\bf A}_x^{\rm T}}{2}+\eta_y\frac{{\bf A}_y+{\bf A}_y^{\rm T}}{2}\notag\\
      \prec &[1+\eta_x s_{1,2}(\theta_3)\Delta x^{\alpha}/2+\eta_y s_{2,2}(\theta_4)\Delta y^{\beta}/2]{\bf I}_J\notag\\
      &+\eta_x(\hat{d}+\theta_3)({\bf I}_{M_y}\otimes {\bf S}_{\alpha,M_x})+\eta_y(\hat{e}+\theta_4)({\bf S}_{\beta,M_y}\otimes {\bf I}_{M_x} )\notag\\
      =&(1+\Delta t b_2){\bf I}_J+\sqrt{\frac{2\hat{d}}{\check{d}}}\bar{d}\eta_x({\bf I}_{M_y}\otimes {\bf S}_{\alpha,M_x})+\sqrt{\frac{2\hat{e}}{\check{e}}}\bar{e}\eta_y({\bf S}_{\beta,M_y}\otimes {\bf I}_{M_x})\notag\\
      \prec& c_2{\bf I}_J+c_2\bar{d}\eta_x({\bf I}_{M_y}\otimes {\bf S}_{\alpha,M_x})+c_2\bar{e}\eta_y({\bf S}_{\beta,M_y}\otimes {\bf I}_{M_x})=c_2\tilde{\bf P},
   	\end{align*}
   	which together with \eqref{halbcontrlbytildp} completes the proof.
   \end{proof}
   
   	\begin{lemma}\label{spartcontrollm}
   	Let ${\bf Z}=\diago(z_1,z_2,...,z_M)\in\mathbb{R}^{M\times M}.$  Let $\tilde{b}>0$ be  such that 
   	\begin{equation*}
   		\nabla({\bf Z})\leq \frac{\tilde{b}}{M+1}.
   	\end{equation*} 
   	Assume there exits $\hat{b}\geq \check{b}>0$ such that
   	\begin{equation*}
   		\hat{b}\geq \max({\bf Z})\geq \min({\bf Z})\geq \check{b}.
   	\end{equation*}
   Then, for any non-zero vector ${\bf y}\in\mathbb{C}^{M\times 1}$, it holds that
   	\begin{align*}
|{\bf y}^{*}({\bf Z}{\bf S}_{\gamma,M}-{\bf S}_{\gamma,M}{\bf Z}){\bf y}|&\leq \nu_1(\hat{b},\check{b}){\bf y}^{*}{\bf S}_{\gamma,M}{\bf y}+\frac{\nu_2(\gamma,\hat{b},\check{b},\tilde{b})}{(1+M)^{\gamma}}{\bf y}^{*}{\bf y},
   	\end{align*}
   	where $\nu_1(\hat{b},\check{b}):=\hat{b}^2+\check{b}^2+1$,
   	\begin{equation*}
   	\nu_2(\gamma,\hat{b},\check{b},\tilde{b}):=\frac{2||\{s_k^{(\gamma)}\}||_{\mathcal{D}_{\gamma}}\hat{b}^2\tilde{b}^2}{\check{b}^2(2-\gamma)}
   	\end{equation*}
   \end{lemma}
   \begin{proof}
   From Lemma \ref{taumatprecspectralm}, we already know that ${\bf S}_{\gamma,M}\succ{\bf O}$. Then,
   \begin{align}
   	|{\bf y}^{*}({\bf Z}{\bf S}_{\gamma,M}-{\bf S}_{\gamma,M}{\bf Z}){\bf y}|&=|\langle{\bf Z}{\bf y},{\bf S}_{\gamma,M}{\bf y}\rangle-\langle{\bf S}_{\gamma,M}{\bf y},{\bf Z}{\bf y}\rangle|\notag\\
   	&\leq 2|\langle{\bf Z}{\bf y},{\bf S}_{\gamma,M}{\bf y}\rangle|\notag\\
   	&=2|\langle{\bf S}_{\gamma,M}^{\frac{1}{2}}{\bf Z}{\bf y},{\bf S}_{\gamma,M}^{\frac{1}{2}}{\bf y}\rangle|\notag\\
   	&\leq 2\langle{\bf S}_{\gamma,M}^{\frac{1}{2}}{\bf Z}{\bf y},{\bf S}_{\gamma,M}^{\frac{1}{2}}{\bf Z}{\bf y}\rangle^{\frac{1}{2}}\langle{\bf S}_{\gamma,M}^{\frac{1}{2}}{\bf y},{\bf S}_{\gamma,M}^{\frac{1}{2}}{\bf y}\rangle^{\frac{1}{2}}\notag\\
   	&\leq \langle{\bf S}_{\gamma,M}^{\frac{1}{2}}{\bf Z}{\bf y},{\bf S}_{\gamma,M}^{\frac{1}{2}}{\bf Z}{\bf y}\rangle+\langle{\bf S}_{\gamma,M}^{\frac{1}{2}}{\bf y},{\bf S}_{\gamma,M}^{\frac{1}{2}}{\bf y}\rangle\notag\\
   	&={\bf y}^{*}{\bf Z}{\bf S}_{\gamma,M}{\bf Z}{\bf y}+{\bf y}^{*}{\bf S}_{\gamma,M}{\bf y}.\label{spartcontroleq1}
   \end{align}
   Moreover,
   \begin{align*}
   	{\bf Z}{\bf S}_{\gamma,M}{\bf Z}&=\frac{1}{2}[{\bf Z}^2{\bf S}_{\gamma,M}+{\bf S}_{\gamma,M}{\bf Z}^2-\Delta_{{\bf S}_{\gamma,M}}({\bf Z}^2)]\notag\\
   	&\preceq \frac{1}{2}({\bf Z}^2{\bf S}_{\gamma,M}+{\bf S}_{\gamma,M}{\bf Z}^2)+\frac{1}{2}||\Delta_{{\bf S}_{\gamma,M}}({\bf Z}^2)||_2{\bf I}_M.
   \end{align*}
   By the assumptions, we have
   \begin{align*}
   	\hat{b}^2&\geq \max({\bf Z}^2)\geq \min({\bf Z}^2)\geq \check{b}^2,\\ 
  \nabla({\bf Z}^2) &\leq \max\limits_{i\neq j}\frac{|z_i^2-z_j^2|}{|i-j|}\\
   &=\max\limits_{i\neq j}\frac{(z_i+z_j)|z_i-z_j|}{|i-j|}\leq 2\hat{b}\max\limits_{i\neq j}\frac{|z_i-z_j|}{|i-j|}= 2\hat{b}\nabla({\bf Z})\leq \frac{2\hat{b}\tilde{b}}{M+1}.
   \end{align*}
   Then, Lemmas \ref{resmatlemm} implies that
   \begin{align*}
   	||\Delta_{{\bf S}_{\gamma,M}}({\bf Z}^2)||_2\leq \frac{\mu_{\gamma}(||\{s_k^{(\gamma)}\}||_{\mathcal{D}_{\gamma}},2\hat{b}\tilde{b},\check{b}^2)}{(1+M)^{\gamma}}
   \end{align*}
   Take $\theta_2=\check{b}^2$. Then, Lemma \ref{hpartcontrollm} implies that
   \begin{align*}
{\bf Z}^2{\bf S}_{\gamma,M}+{\bf S}_{\gamma,M}{\bf Z}^2&\prec 2(\hat{b}^2+\theta_2){\bf S}_{\gamma,M}+\left(\frac{\mu_{\gamma}(||\{s_k^{(\gamma)}\}||_{\mathcal{D}_{\gamma}},2\hat{b}\tilde{b},\theta_2)}{(1+M)^{\gamma}}\right){\bf I}_{M}\\
&=2(\hat{b}^2+\check{b}^2){\bf S}_{\gamma,M}+\left(\frac{\mu_{\gamma}(||\{s_k^{(\gamma)}\}||_{\mathcal{D}_{\gamma}},2\hat{b}\tilde{b},\check{b}^2)}{(1+M)^{\gamma}}\right){\bf I}_{M}.
   \end{align*}
   That means
   \begin{align*}
   	{\bf Z}{\bf S}_{\gamma,M}{\bf Z}&=\frac{1}{2}[({\bf Z}^2{\bf S}_{\gamma,M}+{\bf S}_{\gamma,M}{\bf Z}^2)+||\Delta_{{\bf S}_{\gamma,M}}({\bf Z}^2)||_2{\bf I}_M]\\
   	&=(\hat{b}^2+\check{b}^2){\bf S}_{\gamma,M}+\left(\frac{\mu_{\gamma}(||\{s_k^{(\gamma)}\}||_{\mathcal{D}_{\gamma}},2\hat{b}\tilde{b},\check{b}^2)}{(1+M)^{\gamma}}\right){\bf I}_{M},
   \end{align*}
   which together with \eqref{spartcontroleq1} implies that
   \begin{align*}
   	|{\bf y}^{*}({\bf Z}{\bf S}_{\gamma,M}-{\bf S}_{\gamma,M}{\bf Z}){\bf y}|&\leq {\bf y}^{*}{\bf Z}{\bf S}_{\gamma,M}{\bf Z}{\bf y}+{\bf y}^{*}{\bf S}_{\gamma,M}{\bf y}\\
   	&\leq (\hat{b}^2+\check{b}^2+1){\bf y}^{*}{\bf S}_{\gamma,M}{\bf y}+\left(\frac{\mu_{\gamma}(||\{s_k^{(\gamma)}\}||_{\mathcal{D}_{\gamma}},2\hat{b}\tilde{b},\check{b}^2)}{(1+M)^{\gamma}}\right){\bf y}^{*}{\bf y}\\
   	&=\nu_1(\hat{b},\check{b}){\bf y}^{*}{\bf S}_{\gamma,M}{\bf y}+\frac{\nu_2(\gamma,\hat{b},\check{b},\tilde{b})}{(1+M)^{\gamma}}{\bf y}^{*}{\bf y}.
   \end{align*}
   The proof is complete.
   \end{proof}

   \begin{lemma}\label{2dspartcontrollm}
   	   	 Assume $d\in\mathcal{L}_{1}(\Omega)$, $e\in\mathcal{L}_{2}(\Omega)$. Take $\Delta t\leq \frac{c_3}{b_3}$ with $c_3,b_3>0$ being constants given as follows
   	   	 \begin{align*}
   	   	 	&c_3=\max\left\{\frac{\hat{d}^2+\check{d}^2+1}{2\bar{d}},\frac{\hat{e}^2+\check{e}^2+1}{2\bar{e}}\right\},\\ &b_3=\frac{||\{s_k^{(\alpha)}\}_{k\geq 0}||_{\mathcal{D}_{\alpha}}\hat{d}^2|d|^2_{\mathcal{L}_1(\Omega)}(r_1-l_1)^{2+\alpha}}{\check{d}^2(2-\alpha)}+\frac{||\{s_k^{(\beta)}\}_{k\geq 0}||_{\mathcal{D}_{\beta}}\hat{e}^2|e|^2_{\mathcal{L}_2(\Omega)}(r_2-l_2)^{2+\beta}}{\check{e}^2(2-\beta)}.
   	   	 \end{align*}
   	 Then, for any nonzero vector ${\bf x}\in\mathbb{C}^{J\times 1}$, it holds that
   	   	 \begin{equation*}
   	   	 	\frac{|{\bf x}^{*}\mathcal{S}({\bf A}){\bf x}|}{{\bf x}^{*}\tilde{\bf P}{\bf x}}\leq c_3.
   	   	 \end{equation*}
   \end{lemma}
   \begin{proof}
   	By definition of ${\bf A}$, we have
   	\begin{align*}
   	|{\bf x}^{*}\mathcal{S}({\bf A}){\bf x}|=|\eta_x{\bf x}^{*}\mathcal{S}({\bf A}_x){\bf x}+\eta_y{\bf x}^{*}\mathcal{S}({\bf A}_y){\bf x}|\leq \eta_x|{\bf x}^{*}\mathcal{S}({\bf A}_x){\bf x}|+\eta_y|{\bf x}^{*}\mathcal{S}({\bf A}_y){\bf x}|
   	\end{align*}
   	Rewrite ${\bf x}$ as ${\bf x}=({\bf x}_1;{\bf x}_2;\cdots;{\bf x}_{M_y})$ with ${\bf x}_j\in\mathbb{C}^{M_x\times 1}$ for each $j$. Denote ${\bf D}_{j}={\rm diag}(d_{i,j})_{i=1}^{M_x}$ for $j\in 1\wedge M_y$. Then, for each $j$, it holds that
   	\begin{equation*}
   	\nabla({\bf D}_j)\leq |d|_{\mathcal{L}_1(\Omega)}h_x=\frac{|d|_{\mathcal{L}_1(\Omega)}(r_1-l_1)}{M+1},\quad \hat{d}\geq \max({\bf D}_j)\geq \min({\bf D}_j)\geq \check{d}>0.
   \end{equation*}
   Lemma \ref{spartcontrollm} implies that
   	\begin{align*}
   	\eta_x|{\bf x}^{*}\mathcal{S}({\bf A}_x){\bf x}|&=\frac{	\eta_x}{2}|{\bf x}^{*}[{\bf D}({\bf I}_{M_y}\otimes {\bf S}_{\alpha,M_x})-({\bf I}_{M_y}\otimes {\bf S}_{\alpha,M_x}){\bf D}]{\bf x}|\\
   	&=\frac{	\eta_x}{2}\left|\sum\limits_{j=1}^{M_y}{\bf x}_{j}^{*}({\bf D}_{j}{\bf S}_{\alpha,M_x}-{\bf S}_{\alpha,M_x}{\bf D}_{j}){\bf x}_{j}\right|\\
   	&\leq \frac{	\eta_x}{2}\sum\limits_{j=1}^{M_y}|{\bf x}_{j}^{*}({\bf D}_{j}{\bf S}_{\alpha,M_x}-{\bf S}_{\alpha,M_x}{\bf D}_{j}){\bf x}_{j}|\\
   	&\leq \frac{	\eta_x}{2}\sum\limits_{j=1}^{M_y}\nu_1(\hat{d},\check{d}){\bf x}_{j}^{*}{\bf S}_{\alpha,M_x}{\bf x}_{j}+\frac{\nu_2(\alpha,\hat{d},\check{d},|d|_{\mathcal{L}_1(\Omega)}(r_1-l_1))}{(1+M_x)^{\alpha}}{\bf x}_{j}^{*}{\bf x}_{j}\\
   	&=\frac{\nu_1(\hat{d},\check{d})\eta_x}{2}{\bf x}^{*}({\bf I}_{M_y}\otimes{\bf S}_{\alpha,M_x}){\bf x}+\frac{\nu_2(\alpha,\hat{d},\check{d},|d|_{\mathcal{L}_1(\Omega)}(r_1-l_1))\Delta t}{2(r_1-l_1)^{-\alpha}}{\bf x}^{*}{\bf x}\\
   	&\leq c_3\bar{d}\eta_x{\bf x}^{*}({\bf I}_{M_y}\otimes{\bf S}_{\alpha,M_x}){\bf x}+\frac{\nu_2(\alpha,\hat{d},\check{d},|d|_{\mathcal{L}_1(\Omega)}(r_1-l_1))\Delta t}{2(r_1-l_1)^{-\alpha}}{\bf x}^{*}{\bf x},
   	\end{align*}
   	where $\nu_1$ and $\nu_2$ are defined in Lemma \ref{spartcontrollm}.
    On the other hand, $\eta_y|{\bf x}^{*}\mathcal{S}({\bf A}_y){\bf x}|=\frac{	\eta_y}{2}|{\bf x}^{*}[{\bf E}({\bf S}_{\beta,M_y}\otimes {\bf I}_{M_x} )-({\bf S}_{\beta,M_y}\otimes {\bf I}_{M_x} ){\bf E}]{\bf x}|$. Denote $\tilde{\bf x}={\bf P}_{x\leftrightarrow y}{\bf x}$, $\tilde{\bf E}={\rm diag}(e({\bf V}_{y,x}))$. 
   By applying the permutation matrix ${\bf P}_{x\leftrightarrow y}$, we have
   \begin{align*}
   &	\frac{	\eta_y}{2}|{\bf x}^{*}[{\bf E}({\bf S}_{\beta,M_y}\otimes {\bf I}_{M_x} )-({\bf S}_{\beta,M_y}\otimes {\bf I}_{M_x} ){\bf E}]{\bf x}|\\
   	&=\frac{	\eta_y}{2}|\tilde{\bf x}^{*}[\tilde{\bf E}({\bf I}_{M_x}\otimes {\bf S}_{\beta,M_y}  )-({\bf I}_{M_x}\otimes {\bf S}_{\beta,M_y} )\tilde{\bf E}]\tilde{\bf x}|.
   \end{align*}
   $\frac{	\eta_y}{2}|\tilde{\bf x}^{*}[\tilde{\bf E}({\bf I}_{M_x}\otimes {\bf S}_{\beta,M_y}  )-({\bf I}_{M_x}\otimes {\bf S}_{\beta,M_y} )\tilde{\bf E}]\tilde{\bf x}|$ has a similar structure to $\frac{	\eta_x}{2}|{\bf x}^{*}[{\bf D}({\bf I}_{M_y}\otimes {\bf S}_{\alpha,M_x})-({\bf I}_{M_y}\otimes {\bf S}_{\alpha,M_x}){\bf D}]{\bf x}|$. Repeating the discussion above, one can show that
   \begin{align*}
   	&\frac{	\eta_y}{2}|\tilde{\bf x}^{*}[\tilde{\bf E}({\bf I}_{M_x}\otimes {\bf S}_{\beta,M_y}  )-({\bf I}_{M_x}\otimes {\bf S}_{\beta,M_y} )\tilde{\bf E}]\tilde{\bf x}|\\
   	&\leq \frac{\nu_1(\hat{e},\check{e})\eta_y}{2}\tilde{\bf x}^{*}({\bf I}_{M_x}\otimes{\bf S}_{\beta,M_y})\tilde{\bf x}+\frac{\nu_2(\beta,\hat{e},\check{e},|e|_{\mathcal{L}_2(\Omega)}(r_2-l_2))\Delta t}{2(r_2-l_2)^{-\beta}}\tilde{\bf x}^{*}\tilde{\bf x}\\
   	&=\frac{\nu_1(\hat{e},\check{e})\eta_y}{2}{\bf x}^{*}({\bf S}_{\beta,M_y}\otimes{\bf I}_{M_x}){\bf x}+\frac{\nu_2(\beta,\hat{e},\check{e},|e|_{\mathcal{L}_2(\Omega)}(r_2-l_2))\Delta t}{2(r_2-l_2)^{-\beta}}{\bf x}^{*}{\bf x}\\
   	&\leq c_3\bar{d}\eta_y{\bf x}^{*}({\bf S}_{\beta,M_y}\otimes{\bf I}_{M_x}){\bf x}+\frac{\nu_2(\beta,\hat{e},\check{e},|e|_{\mathcal{L}_2(\Omega)}(r_2-l_2))\Delta t}{2(r_2-l_2)^{-\beta}}{\bf x}^{*}{\bf x}.
   \end{align*}
   That means 
   \begin{align*}
   	|{\bf x}^{*}\mathcal{S}({\bf A}){\bf x}|\leq& \eta_x|{\bf x}^{*}\mathcal{S}({\bf A}_x){\bf x}|+\eta_y|{\bf x}^{*}\mathcal{S}({\bf A}_y){\bf x}|\\
   	\leq &c_3\bar{d}\eta_x{\bf x}^{*}({\bf I}_{M_y}\otimes{\bf S}_{\alpha,M_x}){\bf x}+c_3\bar{e}\eta_y{\bf x}^{*}({\bf S}_{\beta,M_y}\otimes{\bf I}_{M_x}){\bf x}\\
   	&+\frac{\nu_2(\alpha,\hat{d},\check{d},|d|_{\mathcal{L}_1(\Omega)}(r_1-l_1))\Delta t}{2(r_1-l_1)^{-\alpha}}{\bf x}^{*}{\bf x}\\
   	&+\frac{\nu_2(\beta,\hat{e},\check{e},|e|_{\mathcal{L}_2(\Omega)}(r_2-l_2))\Delta t}{2(r_2-l_2)^{-\beta}}{\bf x}^{*}{\bf x}\\
   	=&c_3\bar{d}\eta_x{\bf x}^{*}({\bf I}_{M_y}\otimes{\bf S}_{\alpha,M_x}){\bf x}+c_3\bar{e}\eta_y{\bf x}^{*}({\bf S}_{\beta,M_y}\otimes{\bf I}_{M_x}){\bf x}+b_3\Delta t{\bf x}^{*}{\bf x}\\
   	\leq& c_3\bar{d}\eta_x{\bf x}^{*}({\bf I}_{M_y}\otimes{\bf S}_{\alpha,M_x}){\bf x}+c_3\bar{e}\eta_y{\bf x}^{*}({\bf S}_{\beta,M_y}\otimes{\bf I}_{M_x}){\bf x}+c_3{\bf x}^{*}{\bf x}=c_3{\bf x}^{*}\tilde{\bf P}{\bf x}.
   \end{align*}
   The proof is complete.
   \end{proof}
   
   \begin{theorem}\label{finalthm}
   	Assume $d\in\mathcal{L}_1(\Omega),e\in\mathcal{L}_2(\Omega)$. Take $\Delta t\leq c_{*}$ with $c_{*}$ being a constant independent of discretization step-sizes defined as follows
   	\begin{equation*}
   	c_{*}=\min\left\{c_0,\frac{c_3}{b_3}\right\}>0,
   	\end{equation*}
   	where $c_0,c_3,b_3$ have been defined in Lemmas \ref{prehpartminmaxlm}, \ref{2dspartcontrollm}. 
   	Then, GMRES solver for the preconditioned system \eqref{precsystem} has a linear convergence rate independent of step-sizes, i.e., the residuals generated by (restarted or non-restarted) GMRES solver satisfy
   	\begin{equation*}
   	||{\bf r}_{k}||_2\leq \theta^{k}||\hat{\bf r}_0||_2, \quad k\geq 1,
   	\end{equation*}
   	where ${\bf r}_{k}={\bf P}^{-1}{\bf b}-{\bf P}^{-1}{\bf A}{\bf u}_k$ with ${\bf u}_k$ denoting the $k$-th GMRES iteration for $k\geq 1$; $\hat{\bf r}_0={\bf P}^{-\frac{1}{2}}{\bf b}-{\bf P}^{-\frac{1}{2}}{\bf A}{\bf u}_0$ with ${\bf u}_0\in\mathbb{R}^{J\times 1}$ being an arbitrary initial guess;
   	\begin{equation*}
   	\theta=\sqrt{1-\frac{c_1^2}{3c_1c_2+9c_3^2}}\in(0,1),
   	\end{equation*}
   	is a constant independent of step-sizes; $c_1,c_2,c_3$ have been defined in Lemmas \ref{prehpartminmaxlm}, \ref{2dspartcontrollm} .
   \end{theorem}
   \begin{proof}
   	Denote $\hat{\bf u}_0={\bf P}^{\frac{1}{2}}{\bf u}_0$. Denote $\hat{\mathbf{r}}_{j}:=\mathbf{P}^{-\frac{1}{2}} \mathbf{f}-\mathbf{P}^{-\frac{1}{2}} \mathbf{A} \mathbf{P}^{-\frac{1}{2}}\hat{\mathbf{u}}_{j}$ with ${\mathbf{u}}_{j}$ denoting the $j$-th iterative solution by applying GMRES solver to the auxiliary system \eqref{auxsystem} with $\hat{\bf u}_0$ as initial guess. Then, by Lemma \ref{residual}, we know that
   	\begin{equation}\label{hatrjcontrolrj}
   	||{\bf r}_k||_2\leq ||\hat{\bf r}_k||_2,\quad k\geq 1.
   	\end{equation}
    
   	Denote
   	\begin{equation*}
   		B(x,y,z)=\sqrt{1-\frac{x^2}{xy+z^2}},\quad y\geq x>0,\quad z\geq 0.
   	\end{equation*}
   	Since $\mathcal{H}(\mathbf{P}^{-\frac{1}{2}} \mathbf{A} \mathbf{P}^{-\frac{1}{2}})\succ {\bf O}$, Lemma \ref{gmrescvglm} is applicable to the auxiliary system \eqref{auxsystem}, which means
   	\begin{align*}
   		&||\hat{\bf r}_{k}||_2\leq[B(\lambda_{\min}(\mathcal{H}({\bf P}^{-\frac{1}{2}}{\bf A}  {\bf P}^{-\frac{1}{2}})),\lambda_{\max}(\mathcal{H}({\bf P}^{-\frac{1}{2}}{\bf A}  {\bf P}^{-\frac{1}{2}})),\rho(\mathcal{S}({\bf P}^{-\frac{1}{2}}{\bf A}  {\bf P}^{-\frac{1}{2}})))]^{k}||\hat{\bf r}_{0}||_2,\\
   		& k\geq 1.
   	\end{align*}
   Moreover, it is clear that on the domain of definition, $B(x,y,z)$ is monotonically decreasing with respect to $x\in (0,+\infty)$, and monotonically increasing with respect to $y\in [x,+\infty)$ and to $z\in[0,+\infty)$. By Lemmas \ref{prehpartminmaxlm}, \ref{2dspartcontrollm} and  \eqref{skewpartcontrl1}--\eqref{hpartlbcontrol1}, we know that
   \begin{align*}
   &	\lambda_{\max}(\mathcal{H}({\bf P}^{-\frac{1}{2}}{\bf A}  {\bf P}^{-\frac{1}{2}}))\leq \frac{3c_2}{2},\quad \lambda_{\min}(\mathcal{H}({\bf P}^{-\frac{1}{2}}{\bf A}  {\bf P}^{-\frac{1}{2}}))\geq \frac{c_1}{2}>0\\
   &\rho(\mathcal{S}({\bf P}^{-\frac{1}{2}}{\bf A}  {\bf P}^{-\frac{1}{2}}))\leq \frac{3c_3}{2}.
   \end{align*}
   Therefore, 
   \begin{equation*}
   	||\hat{\bf r}_{k}||_2\leq [B(c_1/2,3c_2/2,3c_3/2)]^{k}||\hat{\bf r}_{0}||_2=\theta^k||\hat{\bf r}_{0}||_2,\quad k\geq 1,
   \end{equation*}
   which together with \eqref{hatrjcontrolrj} implies that
   \begin{equation*}
   	||{\bf r}_k||_2\leq \theta^k||\hat{\bf r}_{0}||_2.
   \end{equation*}
   The proof is complete.
   \end{proof}
   
   \begin{rem}
   	Compared with the assumptions used in \cite{lin2017splitting}, the assumptions in Theorem \ref{finalthm} are mild. The theory in \cite{lin2017splitting} assumes that $d$ and $e$ are both functions of one variable or that $d$ is proportional to $e$, which is in contrast to the theory in this paper only assuming  $d\in\mathcal{L}_1(\Omega)$ and $e\in\mathcal{L}_2(\Omega)$. Actually, an uniformly bounded $|\partial_x d|$ ($|\partial_y e|$, respectively) is sufficient to guarantee $d\in\mathcal{L}_1(\Omega)$ ($e\in\mathcal{L}_2(\Omega)$, respectively). Moreover, since $c_*$ is a constant independent of $\Delta t$, $\Delta x$ and $\Delta y$,  the assumption $\Delta t\leq c_{*}$ is easily satisfied by taking a properly small temporal step-size $\Delta t$. To conclude, Theorem \ref{finalthm} shows that with our proposed $\tau$ preconditioner, the GMRES solver for the preconditioned Toeplitz like system has a convergence rate not deteriorating as the grid get refined, under mild assumptions.
   \end{rem}

\section{The $\tau$ Preconditioner for Multi-dimension SFDE with Variable Coefficient}\label{multidimensionsection}

    Consider the following multi-dimension SFDE with variable coefficients.
    	\begin{align}
    	&(\partial_t u)({\bf x},t)=\sum\limits_{i=1}^{l}d_i({\bf x})(\partial_i^{\alpha_i}u)({\bf x},t)+f({\bf x},t),\qquad({\bf x},t)\in\Omega\times(0,T],\label{mdrsdiffusioneq}\\
    	&u({\bf x},t)=0,\qquad\qquad\qquad \qquad\qquad\qquad\qquad\qquad~~ ({\bf x},t)\in\partial\Omega\times(0,T],\label{mddrcheltboundary}\\
    	&u({\bf x},0) =\psi({\bf x}),\qquad\qquad\qquad\quad\quad\qquad\qquad\qquad\quad {\bf x}\in\bar{\Omega},\label{mdinitialcondition}
    \end{align}
    where $\Omega=\prod\limits_{i=1}^{m}(l_i,r_i)\subset\mathbb{R}^m$ is bounded; $\partial\Omega$ and $\bar{\Omega}$ denotes boundary, closure of $\Omega$, respectively; ${\bf x}=(x_1,x_2,...,x_m)\in\mathbb{R}^m$; $u({\bf x},t)$ is unknown to be solved; $d_i({\bf x})$ is positive over $\Omega\times(0,T]$ with $\hat{d}_i\geq d_i({\bf x})\geq \check{d}_i>0$ for positive constants $\hat{d}_i,\check{d}_i>0,~(1\leq i\leq m)$;
    $f({\bf x},t)$ and $\psi({\bf x})$ are given source term and initial condition, respectively; $\partial_t u$ is the first-order partial derivative of $u$ with respect to $t$; $\alpha_1,\alpha_2,...,\alpha_m\in(1,2)$, 	$\partial_i^{\alpha_i}u$ is the Riesz fractional derivative of order $\alpha_i$ with respect to $x_i$ defined by
    	\begin{equation*}
    	(\partial_i^{\alpha_i}u)({\bf x},t):=\frac{-1}{2\cos(\alpha_i\pi/2)\Gamma(2-\alpha_i)}\frac{\partial^2}{\partial x_i^2}\int_{l_i}^{r_i}\frac{u(x_1,x_2,...,x_{i-1},\xi,x_{i+1},...,x_m,t)}{|x_i-\xi|^{\alpha_i-1}}d\xi,
    \end{equation*}
		
    \subsection{The multi-level Toeplitz like system arising from discretization of \eqref{mdrsdiffusioneq}--\eqref{mdinitialcondition}}
    In this subsection, we present the discretization of \eqref{mdrsdiffusioneq}--\eqref{mdinitialcondition} on uniform grid.
    
    For positive integers $M_i (i=1,2,...,m)$ and $N$, let $\Delta t=T/N$ and $h_i=(b_i-a_i)/(M_i+1) (i=1,2,...,m)$. Denote  $t_n=n\Delta t$ for $n\in 0\wedge N$. Denote $x_{i,j}=a_i+jh_i$, for $0\leq j\leq M_i+1$, $i\in 1\wedge m$.
    Let $\mathbb{N}$ be the set of all integer numbers.
    For $i=1,2,...,m,$ define
    $\mathbb{I}_i=\{j\in\mathbb{N}|1\leq j\leq M_i\}$, $\hat{\mathbb{I}}_i=\{j\in\mathbb{N}|0\leq j\leq M_i+1\}$, $\mathbb{K}=\mathbb{I}_1\times\mathbb{I}_2\times\cdots\times\mathbb{I}_m$, $\hat{\mathbb{K}}=\hat{\mathbb{I}}_1\times\hat{\mathbb{I}}_2\times\cdots\times\hat{\mathbb{I}}_m$, $\partial\mathbb{K}=\hat{\mathbb{K}}\setminus\mathbb{K}$.  For a multiindex $K=(k_1,k_2,...,k_m)\in\hat{\mathbb{K}}$, denote ${\bf x}_K=(x_{1,k_1},x_{2,k_2},...,x_{l,k_m})$. We can then define the set of grid points inside $\Omega$ as follows
    \begin{equation*}
    \mathcal{G}:=\{{\bf x}_K|K\in\mathbb{K}\}.
    \end{equation*}
      Denote by $\mathcal{V}(\mathcal{G})$, the vector obtained from arrange the grid points in $\mathcal{G}$ in a lexicographic ordering \cite{linstbcvg2017}. Then, similar to the derivation of \eqref{toeplitz-likesystem}, the linear system arising from discretization of \eqref{mdrsdiffusioneq}--\eqref{mdinitialcondition} at $n$-th time step can be described as follows \cite{linstbcvg2017}
      \begin{equation}\label{mdtoeplikesys}
      {\bf A}{\bf u}^n={\bf b}^n,\quad n=1,2,...,N,
      \end{equation}
      where ${\bf b}^n={\bf u}^{n-1}+\Delta t{\bf f}^{n}$, ${\bf f}^{n}=f(\mathcal{V}(\mathcal{G}),t_n)$, ${\bf u}^0=\psi(\mathcal{V}(\mathcal{G}))$, the unknown ${\bf u}^n$ is an approximation of $u(\mathcal{V}(\mathcal{G}),t_n)$,
      \begin{align*}
      &{\bf A}={\bf I}_J+\sum\limits_{i=1}^{m}\eta_i{\bf A}_i,\quad \eta_i=\frac{\Delta t}{h_i^{\alpha_i}},\quad {\bf A}_i={\bf D}_i({\bf I}_{M_i^{-}}\otimes{\bf S}_{\alpha_i,M_i}\otimes{\bf I}_{M_i^{+}}),\quad {\bf D}_i={\rm diag}(d_i(\mathcal{V}(\mathcal{G}))),\\
      &M_i^{-}=\prod_{j=1}^{i-1}M_j, ~ M_i^{+}=\prod_{j=i+1}^{m}M_j,~ i\in 2\wedge (m-1),\quad M_1^{-}=M_m^{+}=1,\quad J=\prod\limits_{i=1}^{m}M_i.
      \end{align*}
    \subsection{The $\tau$-preconditioner for the multi-level Toeplitz like system \eqref{mdtoeplikesys} and the implementation}
    Now, our $\tau$-preconditioner for the multi-level Toeplitz like system \eqref{mdtoeplikesys} is defined as
    \begin{equation}\label{mdprecdef}
    {\bf P}:={\bf I}_J+\sum\limits_{i=1}^{m}\eta_i\bar{\bf A}_i,
    \end{equation}
    where
    \begin{equation*}
   \bar{\bf A}_i=\bar{d}_i{\bf I}_{M_i^{-}}\otimes\tau({\bf S}_{\alpha_i,M_i})\otimes{\bf I}_{M_i^{+}},\quad \bar{d}_i=\sqrt{\hat{d}_i\check{d}_i},\quad i\in 1\wedge m.
    \end{equation*}
    
   Instead of solving \eqref{mdtoeplikesys}, we employ the GMRES solver to solve the following preconditioned system
   \begin{equation*}
   	{\bf P}^{-1}{\bf A}{\bf u}^n={\bf P}^{-1}{\bf b}^n,\quad n\in 1\wedge N.
   \end{equation*}
   For ease of statement, we neglect the superscript $n$ appearing in the systems above and use the following preconditioned linear system to represent any one of the linear systems above
    \begin{equation}\label{mdprecsys}
   	{\bf P}^{-1}{\bf A}{\bf u}={\bf P}^{-1}{\bf b},
    \end{equation}
    where ${\bf u}$ and ${\bf b}$ represent ${\bf u}^n$ and ${\bf b}^n$ respectively for some $n\in 1\wedge N$.
    
    ${\bf P}$ defined in \eqref{mdprecdef} is diagonalizable by a multi-dimension sine transform, i.e.,
    \begin{equation}\label{mdpdiagform}
    {\bf P}={\bf Q}{\bf \Lambda}{\bf Q}^{\rm T},
    \end{equation}
    where ${\bf Q}=\bigotimes\limits_{i=1}^{m}{\bf Q}_{M_i}$ with ${\bf Q}_{M_i}$ defined in \eqref{sinematdef} is the so called multi-dimension sine transform; ${\bf \Lambda}$ is a diagonal matrix whose diagonal entries are positive and explicitly known as follows
    \begin{equation*}
    {\bf \Lambda}:={\rm diag}(\lambda_K)_{K\in\mathbb{K}},\quad 	\lambda_K:=1+\sum\limits_{i=1}^{m}\bar{d}_i\eta_i\left(s_0^{(\alpha_i)}+2\sum\limits_{j=1}^{M_{i}-1}s_j^{(\alpha_i)}\cos\left(\frac{\pi K(i)j}{M_i+1}\right)\right).
    \end{equation*}
    
    Clearly, the ${\bf Q}$ defined in \eqref{mdpdiagform} is symmetric and the matrix-vector product associated with ${\bf Q}$ can be computed within $\mathcal{O}(J\log J)$ operations by utilizing  FFTs and the properties of Kronecker product. Then, similar to the implementation discussed in Section \ref{tauprecdef2d}, one can compute the matrix ${\bf P}^{-1}{\bf A}$ defined in \eqref{mdprecsys} times a given vector within $\mathcal{O}(J\log J)$ operations, which is nearly linear.

    \subsection{Convergence of GMRES for the preconditioned multi-level Toeplitz like system \eqref{mdprecsys}}
    	For ${\bf x}=(x_1,x_2,...,x_m),~{\bf y}=(y_1,y_2,...,y_m)\in\Omega$, denote
    \begin{align*}
    	&|{\bf x}-{\bf y}|=\bigg(\sum\limits_{i=1}^{m}|x_i-y_i|^2\bigg)^{\frac{1}{2}},\quad\mathcal{S}_1({\bf x})=(l_1,r_1)\times\left(\prod\limits_{k=2}^{m}\{x_k\}\right),\\
    	&\mathcal{S}_m({\bf x})=\left(\prod\limits_{k=1}^{m-1}\{x_k\}\right)\times(l_m,r_m),\quad \mathcal{S}_i({\bf x})=\prod\limits_{k=1}^{i-1}\{x_k\}\times(l_i,r_i)\times\prod\limits_{k=i+1}^{m}\{x_k\},\quad 2\leq i\leq m-1.
    \end{align*}
    Define the set of functions that are Lipschitz continuous with respect to the spatial variable $x_i~ (i=1,2,...,m)$ as
    \begin{equation*}
    	\mathcal{L}_i( \Omega):=\left\{w:\Omega\rightarrow \mathbb{R} \ \bigg| |w|_{\mathcal{L}_i( \Omega)}:=\sup\limits_{{\bf x}\in\Omega} \ \sup\limits_{{\bf y},{\bf z}\in \mathcal{S}_i({\bf x}),~{\bf y}\neq {\bf z}}\frac{|w({\bf y})-w({\bf z})|}{|{\bf y}-{\bf z}|}<+\infty\right\}.
    \end{equation*}

     Analogous to the proof of Theorem \ref{finalthm}, one can prove the following theorem.
     \begin{theorem}\label{mdfinalthm}
     Assume $d_i\in\mathcal{L}_i(\Omega)$, $i=1,2,...,m$. Take $\Delta t\leq c_*$ with $c_*$ being a constant independent of discretization step-sizes defined as follows
     \begin{equation*}
     c_*=\max\left\{\frac{1-c_1}{b_1},\frac{c_2-1}{b_2},\frac{c_3}{b_3}\right\}>0,
     \end{equation*}
     where
     \begin{align*}
     	&c_1=\min\limits_{i\in 1\wedge m}\sqrt{\frac{\check{d}_i}{2\hat{d}_i}}\in(0,1),\quad c_2=\max\limits_{i\in 1\wedge m}\sqrt{\frac{2\hat{d}_i}{\check{d}_i}}>\sqrt{2},\\
     	&b_1=\sum\limits_{i=1}^{m}\frac{||\{s_k^{(\alpha_i)}\}_{k\geq 0}||_{\mathcal{D}_{\alpha_i}}|d_i|_{\mathcal{L}_{i}(\Omega)}^2(r_i-l_i)^2}{4(1-\sqrt{2}/2)\check{d}_i(2-\alpha_i)},\\
     	&b_2=\sum\limits_{i=1}^{m}\frac{||\{s_k^{(\alpha_i)}\}_{k\geq 0}||_{\mathcal{D}_{\alpha_i}}|d_i|^2_{\mathcal{L}_i(\Omega)}(r_i-l_i)^{2}}{4(\sqrt{2}-1)\hat{d}_i(2-\alpha_i)},\\
     	&c_3=\max\limits_{i\in 1\wedge m}\frac{\hat{d}_i^2+\check{d}_i^2+1}{2\bar{d}_i},\quad i=1,2,...,m,\\
     	&b_3=\sum\limits_{i=1}^{m}\frac{||\{s_k^{(\alpha_i)}\}_{k\geq 0}||_{\mathcal{D}_{\alpha_i}}\hat{d}_i^2|d_i|^2_{\mathcal{L}_i(\Omega)}(r_i-l_i)^{2+\alpha_i}}{\check{d}_i^2(2-\alpha_i)}.
     \end{align*}
     Then,  GMRES solver for the preconditioned system \eqref{mdprecsys} has a linear convergence rate independent of step-sizes, i.e., the residuals generated by GMRES solver satisfy
     \begin{equation*}
     	||{\bf r}_{k}||_2\leq \theta^{k}||\hat{\bf r}_0||_2, \quad k\geq 1,
     \end{equation*}
     where ${\bf r}_{k}={\bf P}^{-1}{\bf b}-{\bf P}^{-1}{\bf A}{\bf u}_k$ with ${\bf u}_k$ denoting the $k$-th GMRES iteration for $k\geq 1$; $\hat{\bf r}_0={\bf P}^{-\frac{1}{2}}{\bf b}-{\bf P}^{-\frac{1}{2}}{\bf A}{\bf u}_0$ with ${\bf u}_0\in\mathbb{R}^{J\times 1}$ being an arbitrary initial guess for \eqref{mdprecsys};
     \begin{equation*}
     	\theta=\sqrt{1-\frac{c_1^2}{3c_1c_2+9c_3^2}}\in(0,1).
     \end{equation*}
     \end{theorem}

    \begin{rem}
    Clearly, the theoretical results presented in Theorem \ref{mdfinalthm} for multi-dimension SFDE  are straightforward extension of that presented in Theorem \ref{finalthm} for 2-dimension SFDE. That demonstrates the power of the proposed preconditioner and the developed theoretical framework in this paper. 
    \end{rem}
    
    \section{Numerical Results}

In this section, we test the proposed preconditioned solver by 2D, 3D examples  and compare its performance with other preconditioning techniques including: the (multilevel) circulant preconditioner proposed in \cite{leisun2013}, the (multilevel) Toeplitz preconditioner proposed in \cite{lin2017splitting}. For ease of statement, we denote by GMRES-${\bf P}$,  GMRES-T and GMRES-C, the preconditioned GMRES methods with the proposed preconditioner, Toeplitz preconditioner and circulant preconditioner, respectively.

    For all preconditioners tested in this section, GMRES solver is applied to solving the preconditioned systems. We set $||{\bf r}_k||_2\leq$1e-7$\times ||{\bf r}_0||_2$  as stopping criterion of GMRES solver, where ${\bf r}_k$ denotes residual vector at $k$-th GMRES iteration and ${\bf r}_0$ denotes the initial residual vector. As we solving the $N$ many unknowns ${\bf u}^n$ ($n=1,2,...,N$) in \eqref{toeplitz-likesystem} sequentially, we take ${\bf u}^0$ as initial guess for solving ${\bf u}^1$ and take ${\bf u}^n$ as initial guess for solving ${\bf u}^{n+1}$ with $n\in 1\wedge (N-1)$ once ${\bf u}^n$ is solved.
    
    As verified in the Appendices A-C, there are three numerical schemes \cite{ccelik2012crank,meerschaert2006finite,sousaelic} appropriate as numerical discretization in this paper. All the numerical results presented in this section are based on the numerical scheme proposed in \cite{ccelik2012crank}, which is defined as follows (see \eqref{centraldiffwkdef})
    \begin{equation*}
    	s_0^{(\gamma)}=\frac{\Gamma(\gamma+1)}{\Gamma(\gamma/2+1)^2},\quad s_{k+1}^{(\gamma)}=\left(1-\frac{\gamma+1}{\gamma/2+k+1}\right)s_{k}^{(\gamma)},\quad k\geq 0,\quad  \gamma\in(1,2).
    \end{equation*}
    
    Numerical results corresponding to the other two schemes are quite similar to the presented one, which are thus neglected considering the limited length of this paper. There are $N$ many linear systems to be solved whose unknowns corresponding to ${\bf u}^{n}$ for $n=1,2,...,N$, respectively. Hence, there are $N$ many iteration numbers for each iterative solver tested. We use the notation  ``iter'' to represent the average of the $N$ many iteration numbers. Denote by `CPU’, the running time of an algorithm. The examples of fractional diffusion equations tested in this section are defined in square physical domains. Each of these square domains is discretized by square grid with $M+1$ uniform partitions along each spatial direction. The time interval $[0,T]$ is always discretized in $N$ many uniform partitions for each example.  As iterative solvers may produce different iterative solutions of the fractional diffusion equations, we use the following norm to measure the accuracy of the iteration solutions
    \begin{equation*}
    	\mathrm{E}_{M,N}=\frac{||{\bf u}_*-\tilde{\bf u}||_{\infty}}{||{\bf u}||_{\infty}},
    \end{equation*}
    where  ${\bf u}_*$ and $\tilde{\bf u}$ denote the exact solution of SFDE and iterative solution deriving from some iterative solver, respectively. As $	\mathrm{E}_{M,N}$ corresponding to different iterative solvers are roughly the same, the results of  $	\mathrm{E}_{M,N}$ are not listed in this paper. In such situation, ``iter'' and ``CPU" are more interesting quantities for measuring the performance of preconditioning techniques.
    
    
    In Example 1, we consider the two-dimension SFDE \eqref{rsdiffusioneq}--\eqref{initialcondition} with
    \begin{align*}
    &\Omega=(0,2)\times(0,2),\quad T=1,\quad d(x,y)=1+x^{\alpha}+(2-x)^{\alpha}+y^{\beta}+(2-y)^{\beta},\\
    &e(x,y)=2+\cos(\pi x/5)+\cos(\pi y/5),\quad \psi(x,y)=x^2(2-x)^2y^2(2-y)^2,\\
    &f(x,y,t)=-\exp(-t)x^2(2-x)^2y^2(2-y)^2\notag\\ &\qquad\qquad\quad+\frac{\exp(-t)d(x,y)y^2(2-y)^2}{2\cos(\alpha\pi/2)\Gamma(2-\alpha)}\sum\limits_{i=2}^{4}\frac{\binom{2}{i-2}2^{4-i}i![x^{i-\alpha}+(2-x)^{i-\alpha}]}{\Gamma(i+1-\alpha)(-1)^{i-2}}\notag\\
    &\qquad\qquad\quad +\frac{\exp(-t)e(x,y)x^2(2-x)^2}{2\cos(\beta\pi/2)\Gamma(2-\beta)}\sum\limits_{i=2}^{4}\frac{\binom{2}{i-2}2^{4-i}i![y^{i-\beta}+(2-y)^{i-\beta}]}{\Gamma(i+1-\beta)(-1)^{i-2}},
    \end{align*}
    the exact solution of which is given by $u(x,y,t)=\exp(-t)x^2(2-x)^2y^2(2-y)^2$. Results of different preconditioned GMRES methods for solving Example 1 
		are listed in Table \ref{expl2dtable}. Table \ref{expl2dtable} shows that the proposed preconditioning method GMRES-${\bf P}$ is the most efficient one in terms of CPU time and iteration number among the three solvers. Moreover, the iteration number of  GMRES-${\bf P}$ changes slightly as $N$ and $M$ changes, which illustrates the size-independent convergence of GMRES-${\bf P}$ as shown in Theorem \ref{finalthm}. 
    
    	\begin{table}[h]
    	\begin{center}
    		\caption{Performance of different preconditioners for solving Example 1.
				}\label{expl2dtable}
    		\setlength{\tabcolsep}{0.48em}
    		\begin{tabular}[c]{ccc|cc|cc|cc}
    			\hline
    			\multirow{2}{*}{$(\alpha,\beta)$} &\multirow{2}{*}{$N$}&\multirow{2}{*}{$M+1$}& \multicolumn{2}{c|}{GMRES-${\bf P}$} &\multicolumn{2}{c|}{GMRES-T}& \multicolumn{2}{c}{GMRES-C}  \\
    			\cline{4-9}
    			&&&$\mathrm{Iter}$&$\mathrm{CPU(s)}$&$\mathrm{Iter}$&$\mathrm{CPU(s)}$&$\mathrm{Iter}$&$\mathrm{CPU(s)}$\\
    			\hline
    			\multirow{9}{*}{(1.1,1.9)}&\multirow{3}{*}{$2^4$} &$2^{8}$   &6.0   &1.43    &7.0 &5.09     &26.4 &4.13  \\
    			                          &                       &$2^{9}$   &6.0   &5.15    &7.0 &21.83    &32.5 &22.15  \\
    			                          &                       &$2^{10}$  &6.0   &18.43   &7.0 &74.62    &39.7 &93.56 \\
    			                          \cline{2-9}
    			                          &\multirow{3}{*}{$2^5$} &$2^{8}$   &5.0   &2.36    &7.0 &9.19     &22.3 &7.98\\
    			                          &                       &$2^{9}$   &6.0   &11.64   &7.0 &39.38    &27.3 &38.33  \\
    			                          &                       &$2^{10}$  &6.0   &45.27   &7.0 &141.99   &33.4 &147.62 \\
    			                          \cline{2-9}
    			                          &\multirow{3}{*}{$2^6$} &$2^{8}$   &5.0   &6.33    &6.0 &19.91    &18.1 &12.78\\
    			                          &                       &$2^{9}$   &5.0   &23.30   &6.0 &72.65    &22.2 &61.17  \\
    			                          &                       &$2^{10}$  &5.0   &79.50   &6.0 &262.91   &28.3 &255.67 \\
    			\hline
    			\multirow{9}{*}{(1.5,1.9)}&\multirow{3}{*}{$2^4$} &$2^{8}$   &6.0   &1.73    &9.0 &6.57     &21.3 &3.74  \\
    			                          &                       &$2^{9}$   &6.0   &6.58    &9.0 &27.34    &25.4 &17.72  \\
    			                          &                       &$2^{10}$  &6.0   &22.34   &9.0 &93.19    &28.7 &64.18 \\
    			                          \cline{2-9}
    			                          &\multirow{3}{*}{$2^5$} &$2^{8}$   &5.0   &3.28    &8.0 &11.82    &19.2 &6.86\\
    			                          &                       &$2^{9}$   &5.0   &11.73   &8.0 &48.36    &22.3 &30.69  \\
    			                          &                       &$2^{10}$  &5.0   &39.86   &8.0 &165.17   &26.3 &119.20 \\
    			                          \cline{2-9}
    			                          &\multirow{3}{*}{$2^6$} &$2^{8}$   &5.0   &6.52    &7.0 &20.21    &15.1 &10.69\\
    			                          &                       &$2^{9}$   &5.0   &23.88   &7.0 &82.62    &18.2 &48.45  \\
    			                          &                       &$2^{10}$  &5.0   &81.11   &7.0 &287.42   &22.2 &195.20 \\
    		    \hline
    		    \multirow{9}{*}{(1.9,1.9)}&\multirow{3}{*}{$2^4$} &$2^{8}$   &5.0   &1.54    &10.0&7.02     &22.4&4.04  \\
    		                              &                       &$2^{9}$   &5.0   &6.11    &10.0&29.37    &26.6&18.46  \\
    		                              &                       &$2^{10}$  &5.0   &20.36   &10.0&101.59   &31.4&72.21 \\
    		                              \cline{2-9}
    		                              &\multirow{3}{*}{$2^5$} &$2^{8}$   &5.0   &3.12    &9.0 &12.78    &21.2&7.65\\
    		                              &                       &$2^{9}$   &5.0   &11.94   &9.0 &52.93    &24.2&33.32  \\
    		                              &                       &$2^{10}$  &5.0   &40.53   &9.0 &183.57   &27.3&123.58 \\
    		                              \cline{2-9}
    		                              &\multirow{3}{*}{$2^6$} &$2^{8}$   &4.0   &5.45    &8.0 &22.86    &17.1&13.03\\
    		                              &                       &$2^{9}$   &4.0   &20.50   &8.0 &94.02    &20.2&58.26  \\
    		                              &                       &$2^{10}$  &4.0   &68.82   &8.0 &325.41   &24.8&229.93 \\
    		   \hline
    		\end{tabular}
    	\end{center}
    \end{table}
  
    
   
    	In Example 2,
			we consider the multi-dimension SFDE \eqref{mdrsdiffusioneq}--\eqref{mdinitialcondition} with 
    	\begin{align*}
    	&m=3,\quad \Omega=(0,1)^3,\quad T=1,\quad d_1({\bf x})=\sum\limits_{i=1}^3x_i^{\alpha_i}(1-x_i)^{\alpha_i},\\
    	&d_2({\bf x})=2+\sum\limits_{i=1}^{3}\cos(\pi x_i/2),\quad d_3({\bf x})=1+x_1x_2x_3,\quad\psi({\bf x})=\prod_{i=1}^{3}x_i^2(1-x_i)^2,\\
    	&f({\bf x},t)=-\exp(-t)\prod_{i=1}^{3}x_i^2(1-x_i)^2\\
    	&\qquad\qquad+\exp(-t)\sum\limits_{i=1}^{3}\frac{d_i({\bf x})x_i^2(1-x_i)^2}{2\cos(\alpha_i\pi/2)\Gamma(2-\alpha_i)}\sum\limits_{j=2}^{4}\frac{\binom{2}{j-2}j![x_i^{j-\alpha_i}+(1-x_i)^{j-\alpha_i}]}{\Gamma(j+1-\alpha_i)(-1)^{j-2}},
    	\end{align*}
    	the exact solution of which is given by
    	\begin{equation*}
    	u({\bf x},t)=\exp(-t)\prod_{i=1}^{3}x_i^2(1-x_i)^2.
    	\end{equation*}
    	In \cite{lin2017splitting}, there is no implementation of GMRES-T discussed for 3-dimension SFDE. Hence, for Example 2,
			we test and compare the performance of GMRES-${\bf P}$ and GMRES-C, the results of which are listed in Table \ref{expl3dtable}. Table \ref{example3d} shows that the proposed preconditioning method GMRES-${\bf P}$ is more efficient than GMRES-C in terms of both iteration number and computational time. Moreover, the iteration number of GMRES-${\bf P}$ shown in Table \ref{example3d} changes slightly as matrix-size change, which demonstrates the size-independent convergence rate, supporting the theoretical results presented in Theorem \ref{mdfinalthm}.
    	
    		\begin{table}[H]
    		\begin{center}
    			\caption{Performance of different preconditioners for solving Example \ref{example2d}.}\label{expl3dtable}
    			\setlength{\tabcolsep}{1.0em}
    			\begin{tabular}[c]{ccc|cc|cc}
    				\hline
    				\multirow{2}{*}{$(\alpha_1,\alpha_2,\alpha_3)$} &\multirow{2}{*}{$N$}&\multirow{2}{*}{$M+1$}& \multicolumn{2}{c|}{GMRES-${\bf P}$} & \multicolumn{2}{c}{GMRES-C}  \\
    				\cline{4-7}
    				&&&$\mathrm{Iter}$&$\mathrm{CPU(s)}$&$\mathrm{Iter}$&$\mathrm{CPU(s)}$\\
    				\hline
    				\multirow{9}{*}{(1.1,1.9,1.5)}&\multirow{3}{*}{$2^1$} &$2^{6}$   &8.0   &1.33    &36.0 &3.36      \\
    				&                       &$2^{7}$   &9.0   &7.79    &47.0 &74.08     \\
    				&                       &$2^{8}$   &9.0   &75.60     &60.0 &938.62   \\
    				\cline{2-7}
    				&\multirow{3}{*}{$2^2$} &$2^{6}$   &8.0   &2.54     &34.0 &14.12     \\
    				&                       &$2^{7}$   &8.0   &18.49    &43.0 &139.89     \\
    				&                       &$2^{8}$   &8.0   &148.72     &55.0 &1658.30    \\
    				\cline{2-7}
    				&\multirow{3}{*}{$2^3$} &$2^{6}$   &7.0   &4.65    &31.0 &24.95    \\
    				&                       &$2^{7}$   &8.0   &36.12    &40.0 &250.30      \\
    				&                       &$2^{8}$   &8.0   &285.95     &51.0 &2725.38   \\
    				\hline
    				\multirow{9}{*}{(1.5,1.1,1.9)}&\multirow{3}{*}{$2^1$} &$2^{6}$   &8.0   &1.31    &26.0 &5.53     \\
    				&                       &$2^{7}$   &8.0   &9.03    &32.0 &48.11     \\
    				&                       &$2^{8}$   &9.0   &78.98     &40.0 &474.40    \\
    				\cline{2-7}
    				&\multirow{3}{*}{$2^2$} &$2^{6}$   &8.0   &2.60    &24.0 &9.50 \\
    				&                       &$2^{7}$   &8.0   &18.20    &30.0 &89.70     \\
    				&                       &$2^{8}$   &8.0   &144.51     &38.0 &878.89  \\
    				\cline{2-7}
    				&\multirow{3}{*}{$2^3$} &$2^{6}$   &7.0   &4.72    &24.0 &17.72   \\
    				&                       &$2^{7}$   &7.0   &32.89    &32.0 &170.46     \\
    				&                       &$2^{8}$   &8.0   &287.44     &38.0 &1416.19   \\
    				\hline
    		  \multirow{9}{*}{(1.9,1.5,1.1)}&\multirow{3}{*}{$2^1$} &$2^{6}$   &8.0   &1.32    &22.0&17.73     \\
    				&                       &$2^{7}$   &9.0   &10.03    &28.0&164.42     \\
    				&                       &$2^{8}$   &9.0   &79.66     &35.0&1479.31   \\
    				\cline{2-7}
    				&\multirow{3}{*}{$2^2$} &$2^{6}$   &8.0   &2.68    &34.0 &5.71   \\
    				&                       &$2^{7}$   &8.0   &18.14    &43.0 &55.01     \\
    				&                       &$2^{8}$   &8.0   &144.75     &56.0 &779.50   \\
    				\cline{2-7}
    				&\multirow{3}{*}{$2^3$} &$2^{6}$   &7.0   &4.83    &31.0 &12.77   \\
    				&                       &$2^{7}$   &8.0   &36.25    &40.0 &126.98      \\
    				&                       &$2^{8}$   &8.0   &281.32     &51.0 &1386.90  \\
    				\hline
    		  \multirow{9}{*}{(1.1,1.5,1.9)}&\multirow{3}{*}{$2^1$} &$2^{6}$   &9.0   &1.43    &29.0&23.71     \\
    				&                       &$2^{7}$   &9.0   &9.93    &36.0&226.11     \\
    				&                       &$2^{8}$   &9.0   &77.03     &46.9&2008.42   \\
    				\cline{2-7}
    				&\multirow{3}{*}{$2^2$} &$2^{6}$   &8.0   &2.52    &27.0 &4.37   \\
    				&                       &$2^{7}$   &8.0   &17.76    &34.0 &36.90     \\
    				&                       &$2^{8}$   &8.0   &141.15     &43.0 &546.48    \\
    				\cline{2-7}
    				&\multirow{3}{*}{$2^3$} &$2^{6}$   &8.0   &5.28    &23.0 &13.78    \\
    				&                       &$2^{7}$   &8.0   &36.20    &29.0 &139.29    \\
    				&                       &$2^{8}$   &8.0   &286.43     &36.0 &1732.65    \\
    				\hline
    			\end{tabular}
    		\end{center}
    	\end{table}
    
    
    \section{Conclusion}
    In this paper, a $\tau$-preconditioner has been proposed for (multilevel) Toeplitz-like linear systems arising from unsteady-state (multi-dimension) SFDE with variable coefficients, the inversion of which can be fast implemented by FSTs. Theoretically, we have shown that with the proposed preconditioner, preconditioned GMRES solver has a convergence rate independent of system-size under mild assumption that the variable coefficients are partially Lipschitz continuous functions. To the best of our knowledge, this is the first iterative solver with size-independent convergence rate for the Toeplitz-like system arising from the variable-coefficients SFDE. Numerical results reported have demonstrated the efficiency of the proposed preconditioning method and supported the theoretical results.
	
	\bibliographystyle{siam}
	\bibliography{myreferences}
	
	\vspace{3mm}
	\begin{appendix}
			\section{Verification of $\{s_k^{(\gamma)}\}_{k\geq 0}$ arising from \cite{ccelik2012crank}}\label{schm1verifysec}
		$\{s_k^{(\gamma)}\}_{k\geq 0}$ arising from \cite{ccelik2012crank} is defined by
		\begin{equation}\label{centraldiffwkdef}
			s_0^{(\gamma)}=\frac{\Gamma(\gamma+1)}{\Gamma(\gamma/2+1)^2},\quad s_{k+1}^{(\gamma)}=\left(1-\frac{\gamma+1}{\gamma/2+k+1}\right)s_{k}^{(\gamma)},\quad k\geq 0.
		\end{equation}
		As $\Gamma(z)>0$ for $z>0$ and $\gamma\in(1,2)$, it is clear that $s_0^{(\gamma)}=\frac{\Gamma(\gamma+1)}{\Gamma(\gamma/2+1)^2}>0$ and that
		\begin{equation*}
			s_{1}^{(\gamma)}=\left(1-\frac{\gamma+1}{\gamma/2+1}\right)s_{0}^{(\gamma)}=\left(\frac{-\gamma}{\gamma+2}\right)s_0^{(\gamma)}<0.
		\end{equation*}
		Notice that $1-\frac{\gamma+1}{\gamma/2+k+1}=\frac{2k-\gamma}{2k+\gamma+2}>0$ for $k\geq 1$. Therefore, it is trivial to see by induction that
		\begin{equation*}
			s_{k+1}^{(\gamma)}=\left(1-\frac{\gamma+1}{\gamma/2+k+1}\right)s_{k}^{(\gamma)}<0,\quad k\geq 1.
		\end{equation*}
		Hence, Property \ref{skprop}${\bf (ii)}$ is valid. Moreover, $1-\frac{\gamma+1}{\gamma/2+k+1}=\frac{2k-\gamma}{2k+\gamma+2}\in(0,1)$ for $k\geq 1$, 
		\begin{equation*}
			|s_{k+1}^{(\gamma)}|=\left|\left(1-\frac{\gamma+1}{\gamma/2+k+1}\right)s_{k}^{(\gamma)}\right|=\left|\frac{2k-\gamma}{2k+\gamma+2}\right||s_{k}^{(\gamma)}|<|s_{k}^{(\gamma)}|, \quad k\geq 1,
		\end{equation*}
		which combined with $s_{k}^{(\gamma)}<0$ for $k\geq 1$ implies that $s_{k}^{(\gamma)}<s_{k+1}^{(\gamma)}$ for $k\geq 1$. In other words, Property \ref{skprop}${\bf (iv)}$ is valid. 
		
		It is indicated in \cite{ccelik2012crank} that $2\sum\limits_{k=1}^{\infty}|s_k^{(\gamma)}|=s_0^{(\gamma)}$, which together with $s_k^{(\gamma)}<0$ for $k\geq 1$ implies that
		\begin{align*}
			s_0^{(\gamma)}+2\sum\limits_{k=1}^{m-1}s_k^{(\gamma)}=-2\sum\limits_{k=1}^{\infty}s_k^{(\gamma)}+2\sum\limits_{k=1}^{m-1}s_k^{(\gamma)}=-2\sum\limits_{k=m}^{\infty}s_k^{(\gamma)}=2\sum\limits_{k=m}^{\infty}|s_k^{(\gamma)}|,\quad m\geq 1
		\end{align*}
		It is also shown in \cite{ccelik2012crank} that $|s_k^{(\gamma)}|=\mathcal{O}((k+1)^{-\gamma-1})$, which means
		\begin{align*}
			s_0^{(\gamma)}+2\sum\limits_{k=1}^{m-1}s_k^{(\gamma)}&=2\sum\limits_{k=m}^{\infty}|s_k^{(\gamma)}|\\
			&\gtrsim \sum\limits_{k=m}^{\infty}\frac{1}{(k+1)^{\gamma+1}}\geq \int_{m}^{\infty}\frac{1}{(1+x)^{1+\gamma}}dx=\frac{1}{\gamma(1+m)^{\gamma}},\quad m\geq 1.
		\end{align*}
		Hence,
		\begin{equation*}
			(m+1)^{\gamma}\left(s_0^{(\gamma)}+2\sum\limits_{k=1}^{m-1}s_k^{(\gamma)}\right)\gtrsim \frac{1}{\gamma}>0,\quad m\geq 1.
		\end{equation*}
		Therefore,
		\begin{equation*}
			\inf\limits_{m\geq 1}(m+1)^{\gamma}\left(s_0^{(\gamma)}+2\sum\limits_{k=1}^{m-1}s_k^{(\gamma)}\right)>0,
		\end{equation*}
		which means Property \ref{skprop}${\bf (iii)}$ is valid.
		
		\section{Verification of $\{s_k^{(\gamma)}\}_{k=0}^{\infty}$ arising from \cite{meerschaert2006finite}}\label{schm2verifysec}
		$\{s_k^{(\gamma)}\}_{k=0}^{\infty}$ arising from \cite{meerschaert2004finite} is defined by 
		\begin{align}
			&s_k^{(\gamma)}=q_{\gamma}\tilde{w}_k^{(\gamma)},\quad k\geq 0,\qquad q_{\gamma}=\frac{-1}{2\cos(\gamma\pi/2)}>0,\label{sftgrwldwk}\\
			&\tilde{w}_k^{(\gamma)}=2g_1^{(\gamma)},\quad \tilde{w}_1^{(\gamma)}=g_0^{(\gamma)}+g_2^{(\gamma)},\quad \tilde{w}_k^{(\gamma)}=g_{k+1}^{(\gamma)},~k\geq 2,\notag\\
			&g_0^{(\gamma)}=-1,\quad g_{k+1}^{(\gamma)}=\left(1-\frac{\gamma+1}{k+1}\right)g_k^{(\gamma)},~k\geq 0.\notag
		\end{align}
		Property \ref{skprop}${\bf (ii)}$, ${\bf (iv)}$ of $\{s_k^{(\gamma)}\}_{k=0}^{\infty}$ defined in \eqref{sftgrwldwk} has been verified in  \cite[Lemma 4.1]{huangxin2022}. 
		
		It thus remains to verify Property \ref{skprop}${\bf (iii)}$. By \cite[Lemma 8]{linstbcvg2017}, it holds
		\begin{equation}\label{gkprop}
			\sum\limits_{k=0}^{\infty}g_k^{(\gamma)}=0.
		\end{equation}
		Therefore,
		\begin{align*}
			s_0^{(\gamma)}+2\sum\limits_{k=1}^{\infty}s_k^{(\gamma)}&=q_{\gamma}\left(\tilde{w}_0^{(\gamma)}+2\sum\limits_{k=1}^{\infty}\tilde{w}_k^{(\gamma)}\right)\\
			&=q_{\gamma}\left[2g_1^{(\gamma)}+2(g_0^{(\gamma)}+g_2^{(\gamma)})+2\sum\limits_{k=2}^{\infty}g_{k+1}^{(\gamma)}\right]\\
			&=2q_{\gamma}\sum\limits_{k=0}^{\infty}g_k^{(\gamma)}=0,
		\end{align*}
		which together with Property \ref{skprop}${\bf (ii)}$ implies that
		\begin{align*}
			s_0^{(\gamma)}+2\sum\limits_{k=1}^{m-1}s_k^{(\gamma)}=-2\sum\limits_{k=1}^{\infty}s_k^{(\gamma)}+2\sum\limits_{k=1}^{m-1}s_k^{(\gamma)}=-2\sum\limits_{m}^{\infty}s_k^{(\gamma)}=2\sum\limits_{k=m}^{\infty}|s_k^{(\gamma)}|
		\end{align*}
		It is shown in \cite{meerschaert2004finite} that $g_k^{(\gamma)}=\mathcal{O}((k+1)^{-\gamma-1})$ for $k\geq 0$,  Therefore, $|s_k^{(\gamma)}|=\mathcal{O}((k+1)^{-\gamma-1})$, which means
		\begin{align*}
			s_0^{(\gamma)}+2\sum\limits_{k=1}^{m-1}s_k^{(\gamma)}&=2\sum\limits_{k=m}^{\infty}|s_k^{(\gamma)}|\\
			&\gtrsim \sum\limits_{k=m}^{\infty}\frac{1}{(k+1)^{\gamma+1}}\geq \int_{m}^{\infty}\frac{1}{(1+x)^{1+\gamma}}dx=\frac{1}{\gamma(1+m)^{\gamma}},\quad m\geq 1.
		\end{align*}
		Hence,
		\begin{equation*}
			(m+1)^{\gamma}\left(s_0^{(\gamma)}+2\sum\limits_{k=1}^{m-1}s_k^{(\gamma)}\right)\gtrsim \frac{1}{\gamma}>0,\quad m\geq 1.
		\end{equation*}
		Therefore,
		\begin{equation*}
			\inf\limits_{m\geq 1}(m+1)^{\gamma}\left(s_0^{(\gamma)}+2\sum\limits_{k=1}^{m-1}s_k^{(\gamma)}\right)>0.
		\end{equation*}
		Thus, Property \ref{skprop}${\bf (iii)}$ of $\{s_k^{(\gamma)}\}_{k=0}^{\infty}$ defined in \eqref{sftgrwldwk}  is valid.
		
		\section{Verification of $\{s_k^{(\gamma)}\}_{k=0}^{\infty}$ arising from \cite{sousaelic}}\label{schm3verifysec}
		$\{s_k^{(\gamma)}\}_{k=0}^{\infty}$ arising from \cite{sousaelic} is defined by 
		\begin{align}
			&s_k^{(\gamma)}=\nu_{\gamma}\hat{w}_k^{(\gamma)},\quad k\geq 0,\qquad \nu_{\gamma}=\frac{-1}{2\cos(\gamma\pi/2)\Gamma(4-\gamma)}>0,\label{wghtfdwk}\\
			&\hat{w}_0^{(\gamma)}=2p_1^{(\gamma)},\quad \hat{w}_1^{(\gamma)}=p_0^{(\gamma)}+p_2^{(\gamma)},\quad \hat{w}_k^{(\gamma)}=p_{k+1}^{(\gamma)},~k\geq 2,\notag\\
			&p_0^{(\gamma)}=-1,\quad p_{1}^{(\gamma)}=4-2^{3-\gamma},\quad p_2^{(\gamma)}=-3^{3-\gamma}+4\times2^{3-\gamma}-6,\notag\\
			&p_k^{(\gamma)}=-(k+1)^{3-\gamma}+4k^{3-\gamma}-6(k-1)^{3-\gamma}+4(k-2)^{3-\gamma}-(k-3)^{3-\gamma},\quad k\geq 3.\notag
		\end{align}
		It is easy to see that $s_0^{(\gamma)}=\nu_{\gamma}\hat{w}_0^{(\gamma)}=2\nu_{\gamma}p_1^{(\gamma)}>0$ and that
		\begin{equation*}
			s_1^{(\gamma)}=\nu_{\gamma}(p_0^{(\gamma)}+p_2^{(\gamma)})=\nu_{\gamma}(-3^{3-\gamma}+4\times2^{3-\gamma}-7)<0,\quad \gamma\in(1,2).
		\end{equation*} 
		Moreover, it is shown in \cite[Lemma 4]{sousaelic} that $p_k^{(\gamma)}\leq 0$ for $k\geq 3$. Then, $s_k^{(\gamma)}=\nu_{\gamma}\hat{w}_k^{(\gamma)}=\nu_{\gamma}p_{k+1}^{(\gamma)}\leq 0$ for $k\geq 2$. So far, Property \ref{skprop}${\bf (ii)}$ of $\{s_k^{(\gamma)}\}_{k=0}^{\infty}$  defined in \eqref{wghtfdwk} is shown to be valid.
		
		Notice that
		\begin{align*}
			\hat{w}_1^{(\gamma)}-\hat{w}_2^{(\gamma)}&=p_0^{(\gamma)}+p_2^{(\gamma)}-p_3^{(\gamma)}\\
			&=4^{3-\gamma}-5\times 3^{3-\gamma}+10\times 2^{3-\gamma}-11\leq 0,\quad \gamma\in(1,2).
		\end{align*}
		In other words, $s_1^{(\gamma)}=\nu_{\gamma}\hat{w}_1^{(\gamma)}\leq \nu_{\gamma}\hat{w}_2^{(\gamma)}=s_2^{(\gamma)}$. Moreover, it is shown in \cite[Lemma  4]{sousaelic} that $p_k^{(\gamma)}\leq p_{k+1}^{(\gamma)}$ for $k\geq 3$. Thus, $s_k^{(\gamma)}=\nu_{\gamma}p_{k+1}^{(\gamma)}\leq\nu_{\gamma}p_{k+3}^{(\gamma)}=s_{k+1}^{(\gamma)}$ for $k\geq 2$, which means  Property \ref{skprop}${\bf (iv)}$ of $\{s_k^{(\gamma)}\}_{k=0}^{\infty}$  defined in \eqref{wghtfdwk} is valid.
		
		It thus remains to verify Property \ref{skprop}${\bf (iii)}$. By \cite[Lemma 4]{sousaelic}, it holds
		\begin{equation}\label{pkprop}
			\sum\limits_{k=0}^{\infty}p_k^{(\gamma)}=0.
		\end{equation}
		Therefore,
		\begin{align*}
			s_0^{(\gamma)}+2\sum\limits_{k=1}^{\infty}s_k^{(\gamma)}&=\nu_{\gamma}\left(\hat{w}_0^{(\gamma)}+2\sum\limits_{k=1}^{\infty}\hat{w}_k^{(\gamma)}\right)\\
			&=\nu_{\gamma}\left[2p_1^{(\gamma)}+2(p_0^{(\gamma)}+p_2^{(\gamma)})+2\sum\limits_{k=2}^{\infty}p_{k+1}^{(\gamma)}\right]\\
			&=2\nu_{\gamma}\sum\limits_{k=0}^{\infty}p_k^{(\gamma)}=0,
		\end{align*}
		which together with Property \ref{skprop}${\bf (ii)}$ implies that
		\begin{align*}
			s_0^{(\gamma)}+2\sum\limits_{k=1}^{m-1}s_k^{(\gamma)}=-2\sum\limits_{k=1}^{\infty}s_k^{(\gamma)}+2\sum\limits_{k=1}^{m-1}s_k^{(\gamma)}=-2\sum\limits_{m}^{\infty}s_k^{(\gamma)}=2\sum\limits_{k=m}^{\infty}|s_k^{(\gamma)}|
		\end{align*}
		It is shown in \cite{linstbcvg2017} that $p_k^{(\gamma)}=\mathcal{O}((k+1)^{-\gamma-1})$ for $k\geq 0$,  Therefore, $|s_k^{(\gamma)}|=\mathcal{O}((k+1)^{-\gamma-1})$, which means
		\begin{align*}
			s_0^{(\gamma)}+2\sum\limits_{k=1}^{m-1}s_k^{(\gamma)}&=2\sum\limits_{k=m}^{\infty}|s_k^{(\gamma)}|\\
			&\gtrsim \sum\limits_{k=m}^{\infty}\frac{1}{(k+1)^{\gamma+1}}\geq \int_{m}^{\infty}\frac{1}{(1+x)^{1+\gamma}}dx=\frac{1}{\gamma(1+m)^{\gamma}},\quad m\geq 1.
		\end{align*}
		Hence,
		\begin{equation*}
			(m+1)^{\gamma}\left(s_0^{(\gamma)}+2\sum\limits_{k=1}^{m-1}s_k^{(\gamma)}\right)\gtrsim \frac{1}{\gamma}>0,\quad m\geq 1.
		\end{equation*}
		Therefore,
		\begin{equation*}
			\inf\limits_{m\geq 1}(m+1)^{\gamma}\left(s_0^{(\gamma)}+2\sum\limits_{k=1}^{m-1}s_k^{(\gamma)}\right)>0.
		\end{equation*}
		Thus, Property \ref{skprop}${\bf (iii)}$ of $\{s_k^{(\gamma)}\}_{k=0}^{\infty}$ defined in \eqref{wghtfdwk}  is valid.

	\section{Proof of Lemma \ref{taumatprecspectralm}}\label{taumatpreclmproof}
	 \begin{proof}
		Denote ${\bf H}_{\gamma,m}:={\bf S}_{\gamma,m}-\tau({\bf S}_{\gamma,m})$. Rewrite ${\bf H}_{\gamma,m}$ as ${\bf H}_{\gamma,m}=[h_{ij}]_{i,j=1}^{m}$. Then, straightforward calculation yields that
		\begin{equation*}
			h_{ij}=\begin{cases}
				s_{i+j}^{(\gamma)},\quad i+j<m-1,\\
				s_{2m+2-(i+j)},\quad i+j> m+1,\\
				0,\quad  {\rm otherwise}.
			\end{cases}
		\end{equation*}
		By Property \ref{skprop}${\bf (ii)}$, we know that 
		\begin{equation}\label{hijneg}
			h_{ij}\leq 0,\quad 1\leq i,j\leq m.
		\end{equation}
		Denote $p_{ij}=s_{|i-j|}^{(\gamma)}-h_{ij}$.		Then, \eqref{hijneg} and Property \ref{skprop}${\bf (ii)}$,${\bf (iv)}$ imply that
		\begin{equation}\label{pijsign}
			p_{ij}=\begin{cases}
				s_{0}^{(\gamma)}-h_{ii}>0,\quad i=j,\\
				s_{|i-j|}^{(\gamma)}-s_{i+j}^{(\gamma)}\leq 0,\quad i+j<m-1{\rm~and~} i\neq j,\\
				s_{|i-j|}^{(\gamma)}-s_{2m+2-(i+j)}^{(\gamma)}\leq 0,\quad i+j> m+1{\rm~and~} i\neq j,\\
				s_{|i-j|}^{(\gamma)}\leq 0,\quad {\rm otherwise}.
			\end{cases}
		\end{equation}
		Let $(\lambda,{\bf z})$ be an eigen-pair of $\tau({\bf S}_{\gamma,m})^{-1}{\bf H}_{\gamma,m}$ such that $||{\bf z}||_{\infty}=1$. Then, 
		\begin{equation}\label{tauwinvweigpair}
			{\bf H}_{\gamma,m} {\bf z} = \lambda \tau({\bf S}_{\gamma,m}){\bf z}. 
		\end{equation}
		Rewrite ${\bf z}$ as ${\bf z}=(z_1,z_2,...,z_m)^{\rm T}$. 
		\eqref{tauwinvweigpair} implies that for each $i=1,2,...,m$, it holds
		\begin{equation*}
			\sum_{j=1}^{m}h_{ij}z_j=\lambda \sum_{j=1}^{m}p_{ij}z_j.
		\end{equation*}
		Hence,
		$$\lambda p_{ii}z_i=\sum_{j=1}^{m}h_{ij}z_j-\lambda \sum_{j=1,i\neq j}^{m}p_{ij}z_j,\quad i=1,2,...,m.$$
		As $||{\bf z}||_{\infty}=1$, there exists  $k_0\in\{1,2,...,m\}$ such that $|z_{k_0}|=1$. Then,
		\begin{equation}\label{lambdapkkbd}
			|\lambda| |p_{kk}|\leq \sum_{j=1}^{m}|h_{kj}|+|\lambda| \sum_{j=1,k\neq j}^{m}|p_{kj}|.
		\end{equation}
		By \eqref{pijsign} and Property \ref{skprop}${\bf (iii)}$, we have
		\begin{align*}
			&|p_{kk}|-\sum_{j=1,j\neq k}^{m}|p_{kj}|-2\sum_{j=1}^{m}|h_{kj}| \\
			=&(s_0^{(\gamma)}-h_{kk})- \sum_{j=1,j\neq k}^{m}(h_{kj}-s_{|k-j|}^{(\gamma)})+2\sum_{j=1}^{m}h_{kj}\\
			=& s_0^{(\gamma)}+\sum_{j=1,j\neq k}^{m}s_{|k-j|}^{(\gamma)}+\sum_{j=1}^{m}h_{kj} \\
			=& s_0^{(\gamma)}+\left(\sum_{j=1}^{k-1}s_{j}^{(\gamma)}+\sum_{j=1}^{m-k}s_{j}^{(\gamma)}\right)+\left(\sum_{j=k+1}^{m-1}s_{j}^{(\gamma)}+\sum_{j=m-k+2}^{m-1}s_{j}^{(\gamma)}\right)\\
			\geq&s_0^{(\gamma)}+2\sum_{j=1}^{m-1}s_{j}^{(\gamma)}  > 0,
		\end{align*} 
		which combined with \eqref{lambdapkkbd} implies that
		$$|\lambda|\leq \frac{ \sum\limits_{j=1}^{m}|h_{kj}|}{|p_{kk}|-\sum\limits_{j=1,k\neq j}^{m}|p_{kj}|}\leq  \frac{ \sum\limits_{j=1}^{m}|h_{kj}|}{2\sum\limits_{j=1}^{m}|h_{kj}|} =\frac{1}{2}. $$
		Therefore,
		\begin{equation}\label{tauinvhspectr}
			\sigma(\tau({\bf S}_{\gamma,m})^{-1}{\bf H}_{\gamma,m})\subset(-1/2,1/2).
		\end{equation}
		
		Since	$	\tau({\bf S}_{\gamma,m})={\bf S}_{\gamma,m}-{\bf H}_{\gamma,m}$, $$\tau({\bf S}_{\gamma,m})^{-1}{\bf S}_{\gamma,m}=\tau({\bf S}_{\gamma,m})^{-1}(\tau({\bf S}_{\gamma,m})+{\bf H}_{\gamma,m})={\bf I}_m+\tau({\bf S}_{\gamma,m})^{-1}{\bf H}_{\gamma,m},$$
		which means $\sigma(\tau({\bf S}_{\gamma,m})^{-1}{\bf S}_{\gamma,m})=1+\sigma(\tau({\bf S}_{\gamma,m})^{-1}{\bf H}_{\gamma,m})\subset(1/2,3/2)$. By Lemma \ref{tausgammamhpdlm}, $\tau({\bf S}_{\gamma,m})\succ {\bf O}$. By matrix similarity, we have $\sigma(\tau({\bf S}_{\gamma,m})^{-\frac{1}{2}}{\bf S}_{\gamma,m}\tau({\bf S}_{\gamma,m})^{-\frac{1}{2}})=\sigma(\tau({\bf S}_{\gamma,m})^{-1}{\bf S}_{\gamma,m})\subset(1/2,3/2)$, which implies that
		\begin{equation*}
			{\bf O}\prec \frac{1}{2}\tau({\bf S}_{\gamma,m})\prec {\bf S}_{\gamma,m}\prec \frac{3}{2}\tau({\bf S}_{\gamma,m}).
		\end{equation*}
		The proof is complete.
	\end{proof}
	\section{Proof of Lemma \ref{resmatlemm}}\label{retmatlmproof}
	\begin{proof}
		Recall that $\{s_k^{(\gamma)}\}_{k\geq 0}\in\mathcal{D}_{\gamma},~ \forall \gamma\in(1,2)$.	It is easy to check the ($i,j)$-th entry $r_{ij}$ of $\Delta_{{\bf S}_{\gamma,M}}({\bf Z})$ is given by
		\begin{equation*}
			r_{ij}=\left(z_{i}s_{|i-j|}^{({\gamma})}+s_{|i-j|}^{({\gamma})}z_{j}-2z_{i}^{\frac{1}{2}}s_{|i-j|}^{({\gamma})}z_{j}^{\frac{1}{2}}\right)=\Big(z_{i}^{\frac{1}{2}}-z_{j}^{\frac{1}{2}}\Big)^2s_{|i-j|}^{({\gamma})}, \qquad 1\leq i,j\in 1\wedge M.
		\end{equation*}
		Using assumptions ${\bf (i)}$--${\bf (iii)}$, it holds
		\begin{align}
			|r_{ij}| = \Big|s_{|i-j|}^{({\gamma})}\Big|\Big|z_{i}^{\frac{1}{2}}-z_{j}^{\frac{1}{2}}\Big|^2 &=  \Big|s_{|i-j|}^{({\gamma})}\Big|\left|\int_{z_{i}}^{z_{j}}\frac{1}{2}\xi^{-\frac{1}{2}}d\xi\right|^2\notag\\
			& \leq  \Big|s_{|i-j|}^{({\gamma})}\Big|\left|\int_{z_{i}}^{z_{j}}\frac{1}{2}\check{b}^{-\frac{1}{2}}d\xi\right|^2
			\notag\\
			&=  4^{-1}\check{b}^{-1} |s_{|i-j|}^{({\gamma})}||z_{i}-z_{j}|^2\nonumber \\
			&\leq \frac{|s_{|i-j|}^{({\gamma})}|\tilde{b}^2|i-j|^2}{4\check{b}(M+1)^2}\notag\\
			&\leq \frac{||\{s_{k}^{({\gamma})}\}||_{\mathcal{D}_{{\gamma}}}\tilde{b}^2|i-j|^2}{4\check{b}(1+M)^{2}(1+|i-j|)^{{\gamma}+1}},\notag
		\end{align}
		which implies that
		\begin{align*}
			||\Delta_{{\bf S}_{\gamma,M}}({\bf Z})
			||_{\infty}&=\max\limits_{1\leq i\leq M}\sum\limits_{j=1}^{M}|r_{ij}|\notag\\
			&=\max\limits_{1\leq i\leq M}\Big(|r_{ii}|+\sum\limits_{j=1}^{i-1}|r_{ij}|+\sum\limits_{j=i+1}^{M}|r_{ij}|\Big)\notag\\
			&\leq\max\limits_{1\leq i\leq M}\frac{||\{s_{k}^{({\gamma})}\}||_{\mathcal{D}_{{\gamma}}}\tilde{b}^2}{4\check{b}(M+1)^{2}}\left(\sum\limits_{j=1}^{i-1}\frac{|i-j|^2}{(1+|i-j|)^{{\gamma}+1}}+\sum\limits_{j=i+1}^{M}\frac{|i-j|^2}{(1+|i-j|)^{{\gamma}+1}}\right)\notag\\
			&=\max\limits_{1\leq i\leq M}\frac{||\{s_{k}^{({\gamma})}\}||_{\mathcal{D}_{{\gamma}}}\tilde{b}^2}{4\check{b}(M+1)^{2}}\left(\sum\limits_{k=1}^{i-1}\frac{k^2}{(1+k)^{{\gamma}+1}}+\sum\limits_{k=1}^{M-i}\frac{k^2}{(1+k)^{{\gamma}+1}}\right)\notag\\
			&\leq\frac{\tilde{b}^2||\{s_{k}^{({\gamma})}\}||_{\mathcal{D}_{{\gamma}}}}{2\check{b}(M+1)^{2}}\sum\limits_{k=1}^{M}k^{1-{\gamma}}\notag\\
			&\leq\frac{\tilde{b}^2||\{s_{k}^{({\gamma})}\}||_{\mathcal{D}_{{\gamma}}}}{2\check{b}(M+1)^{2}}\sum\limits_{k=1}^{M}\int_{k-1}^{k}x^{1-{\gamma}}dx\notag\\
			&=\frac{||\{s_{k}^{({\gamma})}\}||_{\mathcal{D}_{{\gamma}}}\tilde{b}^2M^{2-{\gamma}}}{2\check{b}(M+1)^{2}(2-{\gamma})}\leq\frac{||\{s_{k}^{({\gamma})}\}||_{\mathcal{D}_{{\gamma}}}\tilde{b}^2}{2\check{b}(2-{\gamma})(M+1)^{\gamma}}.
		\end{align*}
		Since ${\bf S}_{\gamma,M}$ is symmetric, $\Delta_{{\bf S}_{\gamma,M}}({\bf Z})$ is also symmetric.
		Therefore, we have
		$||\Delta_{{\bf S}_{\gamma,M}}({\bf Z})||_2=\rho(\Delta_{{\bf S}_{\gamma,M}}({\bf Z}))\leq||\Delta_{{\bf S}_{\gamma,M}}({\bf Z})||_{\infty}$, which together with the above inequality completes the proof.
	\end{proof}
	\section{Proof of Lemma \ref{hpartcontrollm}}\label{hpartcontrllmproof}
	 \begin{proof}
		We note that $\nabla(\cdot)$ is shift-invariant. i.e.,
		\begin{equation*}
			\nabla({\bf Z}-\theta_1\check{b}{\bf I}_{M})= \nabla({\bf Z})\leq \frac{\tilde{b}}{M+1}.
		\end{equation*}
		Moreover, $\min({\bf Z}-\theta_1\check{b}{\bf I}_{M})\geq (1-\theta_1)\check{b}>0$. By Lemma \ref{resmatlemm}, it holds that
		\begin{equation*}
			||\Delta_{{\bf S}_{\gamma,M}}({\bf Z}-\theta_1\check{b}{\bf I}_{M})||_2\leq\frac{\mu_{\gamma}(||\{s_{k}^{({\gamma})}\}||_{\mathcal{D}_{{\gamma}}},\tilde{b},(1-\theta_1)\check{b})}{(1+M)^{\gamma}},
		\end{equation*}
		where the function $\mu_{\gamma}(\cdot,\cdot,\cdot)$ is defined in Lemma \ref{resmatlemm}.
		
		Moreover, ${\bf O}\prec({\bf Z}-\theta_1\check{b}{\bf I}_{M})^{\frac{1}{2}}$ and ${\bf O}\prec{\bf S}_{\gamma,M}$ imply that
		${\bf O}\prec({\bf Z}-\theta_1\check{b}{\bf I}_{M})^{\frac{1}{2}}{\bf S}_{\gamma,M}({\bf Z}-\theta_1\check{b}{\bf I}_{M})^{\frac{1}{2}}$.
		Hence,
			\begin{align}
				{\bf Z}{\bf S}_{\gamma,M}+{\bf S}_{\gamma,M}{\bf Z}&=2\theta_1\check{b}{\bf S}_{\gamma,M}+({\bf Z}-\theta_1\check{b}{\bf I}_{M}){\bf S}_{\gamma,M}+{\bf S}_{\gamma,M}({\bf Z}-\theta_1\check{b}{\bf I}_{M})\notag\\
				&=2\theta_1\check{b}{\bf S}_{\gamma,M}+2({\bf Z}-\theta_1\check{b}\check{b}{\bf I}_{M})^{\frac{1}{2}}{\bf S}_{\gamma,M}({\bf Z}-\theta_1\check{b}{\bf I}_{M})^{\frac{1}{2}}+\Delta_{{\bf S}_{\gamma,M}}({\bf Z}-\theta_1\check{b}{\bf I}_{M})\notag\\
				&\succ 2\theta_1\check{b}{\bf S}_{\gamma,M}-||\Delta_{{\bf S}_{\gamma,M}}({\bf Z}-\theta_1\check{b}{\bf I}_{M})||_2{\bf I}_M\notag\\
				&\succeq 2\theta_1\check{b}{\bf S}_{\gamma,M}-\frac{\mu_{\gamma}(||\{s_{k}^{({\gamma})}\}||_{\mathcal{D}_{{\gamma}}},\tilde{b},(1-\theta_1)\check{b})}{(1+M)^{\gamma}}{\bf I}_M.\label{sddslowbd}
			\end{align}
			
		Again, shift-invariance of $\nabla(\cdot)$ implies that
		\begin{equation*}
			\nabla((\hat{b}+\theta_2){\bf I}_{M}-{\bf Z})=\nabla({\bf Z})\leq \frac{\tilde{b}}{1+M}.
		\end{equation*}
		Besides, $\min((\hat{b}+\theta_2){\bf I}_{M}-{\bf Z})\geq \theta_2>0.$
		Lemma \ref{resmatlemm} implies that
		\begin{equation*}
			||\Delta_{{\bf S}_{\gamma,M}}((\hat{b}+\theta_2){\bf I}_{M}-{\bf Z})||_2\leq\frac{\mu_{\gamma}(||\{s_{k}^{({\gamma})}\}||_{\mathcal{D}_{{\gamma}}},\tilde{b},\theta_2)}{(1+M)^{\gamma}}.
		\end{equation*}
		Moreover, ${\bf O}\prec((\hat{b}+\theta_2){\bf I}_{M}-{\bf Z})^{\frac{1}{2}}$ and ${\bf O}\prec{\bf S}_{\gamma,M}$ imply that
		$$
		{\bf O}\prec[(\hat{b}+\theta_2){\bf I}_{M}-{\bf Z}]^{\frac{1}{2}}{\bf S}_{\gamma,M}[(\hat{b}+\theta_2){\bf I}_{M}-{\bf Z}]^{\frac{1}{2}}.
		$$
		Hence,
		\begin{align*}
			2(\hat{b}+\theta_2){\bf S}_{\gamma,M}=&((\hat{b}+\theta_2){\bf I}_{M}-{\bf Z}){\bf S}_{\gamma,M}+{\bf S}_{\gamma,M}((\hat{b}+\theta_2){\bf I}_{M}-{\bf Z})+{\bf Z}{\bf S}_{\gamma,M}+{\bf S}_{\gamma,M}{\bf Z}\\
			=&{\bf Z}{\bf S}_{\gamma,M}+{\bf S}_{\gamma,M}{\bf Z}+2[(\hat{b}+\theta_2){\bf I}_{M}-{\bf Z}]^{\frac{1}{2}}{\bf S}_{\gamma,M}[(\hat{b}+\theta_2){\bf I}_{M}-{\bf Z}]^{\frac{1}{2}}\\
			&+\Delta_{{\bf S}_{\gamma,M}}((\hat{b}+\theta_2){\bf I}_{M}-{\bf Z})\\
			\succ& {\bf Z}{\bf S}_{\gamma,M}+{\bf S}_{\gamma,M}{\bf Z}+||\Delta_{{\bf S}_{\gamma,M}}((\hat{b}+\theta_2){\bf I}_{M}-{\bf Z})||_2{\bf I}_M\\
			\succeq& {\bf Z}{\bf S}_{\gamma,M}+{\bf S}_{\gamma,M}{\bf Z}-\frac{\mu_{\gamma}(||\{s_{k}^{({\gamma})}\}||_{\mathcal{D}_{{\gamma}}},\tilde{b},\theta_2)}{(1+M)^{\gamma}}{\bf I}_{M}.
		\end{align*}
		That means
		\begin{equation*}
			2(\hat{b}+\theta_2){\bf S}_{\gamma,M}+\frac{\mu_{\gamma}(||\{s_{k}^{({\gamma})}\}||_{\mathcal{D}_{{\gamma}}},\tilde{b},\theta_2)}{(1+M)^{\gamma}}{\bf I}_{M}\succ {\bf Z}{\bf S}_{\gamma,M}+{\bf S}_{\gamma,M}{\bf Z},
		\end{equation*}
		which together with  \eqref{sddslowbd} completes the proof.
	\end{proof}
	\section{Proof of Lemma \ref{2dhpartcontrollm}}\label{d2hpartcontrollmproof}
	 \begin{proof}
		Denote
		\begin{equation*}
			{\bf D}_{j}={\rm diag}(d_{i,j})_{i=1}^{M_x},\quad 1\leq j\leq M_y.
		\end{equation*}
		Then, it is straightforward to check that
		\begin{align*}
			{\bf A}_x+{\bf A}_x^{\rm T}&={\bf D}({\bf I}_{M_y}\otimes {\bf S}_{\alpha,M_x})+({\bf I}_{M_y}\otimes {\bf S}_{\alpha,M_x}){\bf D}\\
			&={\rm blockdiag}({\bf D}_j{\bf S}_{\alpha,M_x}+{\bf S}_{\alpha,M_x}{\bf D}_j)_{j=1}^{M_y}.
		\end{align*}
		For each fixed $j\in 1\wedge M_y$, it holds that
		\begin{equation*}
			\nabla({\bf D}_j)\leq |d|_{\mathcal{L}_1(\Omega)}h_x=\frac{|d|_{\mathcal{L}_1(\Omega)}(r_1-l_1)}{M+1},\quad \hat{d}\geq \max({\bf D}_j)\geq \min({\bf D}_j)\geq \check{d}>0.
		\end{equation*}
		Hence, Lemma \ref{hpartcontrollm} implies that
		\begin{align*}
			&2\theta_1\check{d}{\bf S}_{\alpha,M_x}-s_{1,1}(\theta_1)\Delta x^{\alpha}{\bf I}_{M_x}\\
			&=2\theta_1\check{d}{\bf S}_{\alpha,M_x}-\left(\frac{\mu_{\alpha}(||\{s_{k}^{({\alpha})}\}||_{\mathcal{D}_{{\alpha}}},|d|_{\mathcal{L}_1(\Omega)}(r_1-l_1),(1-\theta_1)\check{d})}{(1+M)^{\alpha}}\right){\bf I}_{M_x}\\
			&\preceq {\bf D}_j{\bf S}_{\alpha,M_x}+{\bf S}_{\alpha,M_x}{\bf D}_j\\
			&\preceq  2(\hat{d}+\theta_3){\bf S}_{\alpha,M_x}+\left(\frac{\mu_{\alpha}(||\{s_{k}^{({\alpha})}\}||_{\mathcal{D}_{{\alpha}}},|d|_{\mathcal{L}_1(\Omega)},\theta_3)}{(1+M)^{\alpha}}\right){\bf I}_{M_x}\\
			&=2(\hat{d}+\theta_3){\bf S}_{\alpha,M_x}+s_{1,2}(\theta_3)\Delta x^{\alpha}{\bf I}_{M_x},
		\end{align*}
		which means
		\begin{equation*}
			2\theta_1\check{d}({\bf I}_{M_y}\otimes {\bf S}_{\alpha,M_x})-s_{1,1}(\theta_1)\Delta x^{\alpha}{\bf I}_{J}\prec \underbrace{{\rm blockdiag}({\bf D}_j{\bf S}_{\alpha,M_x}+{\bf S}_{\alpha,M_x}{\bf D}_j)_{j=1}^{M_y}}_{:={\bf A}_x+{\bf A}_x^{\rm T}}\prec 2(\hat{d}+\theta_3)({\bf I}_{M_y}\otimes {\bf S}_{\alpha,M_x})+s_{1,2}(\theta_3)\Delta x^{\alpha}{\bf I}_{J}.
		\end{equation*}
		
		On the other hand, by applying the permutation matrix,  one can check that
		\begin{equation*}
			{\bf A}_y+{\bf A}_y^{\rm T}={\bf P}_{x\leftrightarrow y}^{\rm T}[\tilde{\bf E}({\bf I}_{M_x}\otimes {\bf S}_{\beta,M_y})+({\bf I}_{M_x}\otimes {\bf S}_{\beta,M_y})\tilde{\bf E}]{\bf P}_{x\leftrightarrow y},
		\end{equation*}
		with $\tilde{\bf E}={\rm diag}(e({\bf V}_{y,x}))$. Clearly, $\tilde{\bf E}({\bf I}_{M_x}\otimes {\bf S}_{\beta,M_y})+({\bf I}_{M_x}\otimes {\bf S}_{\beta,M_y})\tilde{\bf E}$ has a similar structure with 
		${\bf A}_x+{\bf A}_x^{\rm T}={\bf D}({\bf I}_{M_y}\otimes {\bf S}_{\alpha,M_x})+({\bf I}_{M_y}\otimes {\bf S}_{\alpha,M_x}){\bf D}$. One can repeat the discussion for ${\bf A}_x+{\bf A}_x^{\rm T}$ to show that
		\begin{equation*}
			2\theta_2\check{e}( {\bf I}_{M_x}\otimes {\bf S}_{\beta,M_y})-s_{2,1}(\theta_2)\Delta y^{\beta}{\bf I}_{J}\prec\tilde{\bf E}({\bf I}_{M_x}\otimes {\bf S}_{\beta,M_y})+({\bf I}_{M_x}\otimes {\bf S}_{\beta,M_y})\tilde{\bf E}\prec 2(\hat{e}+\theta_4)({\bf I}_{M_x}\otimes  {\bf S}_{\beta,M_y})+s_{2,2}(\theta_4)\Delta y^{\beta}{\bf I}_{J}.
		\end{equation*}
		Then,
		\begin{align*}
			2\theta_2\check{e}( {\bf S}_{\beta,M_y}\otimes {\bf I}_{M_x})-s_{2,1}(\theta_2)\Delta y^{\beta}{\bf I}_{J}&={\bf P}_{x\leftrightarrow y}^{\rm T}[2\theta_2\check{e}( {\bf I}_{M_x}\otimes {\bf S}_{\beta,M_y})-s_{2,1}(\theta_2)\Delta y^{\beta}{\bf I}_{J}]{\bf P}_{x\leftrightarrow y}\\
			&\prec\underbrace{{\bf P}_{x\leftrightarrow y}^{\rm T}[\tilde{\bf E}({\bf I}_{M_x}\otimes {\bf S}_{\beta,M_y})+({\bf I}_{M_x}\otimes {\bf S}_{\beta,M_y})\tilde{\bf E}]{\bf P}_{x\leftrightarrow y}}_{:={\bf A}_y+{\bf A}_y^{\rm T}}\\
			&\prec {\bf P}_{x\leftrightarrow y}^{\rm T}[2(\hat{e}+\theta_4)({\bf I}_{M_x}\otimes  {\bf S}_{\beta,M_y})+s_{2,2}(\theta_4)\Delta y^{\beta}{\bf I}_{J}] {\bf P}_{x\leftrightarrow y}\\
			&=2(\hat{e}+\theta_4)( {\bf S}_{\beta,M_y}\otimes{\bf I}_{M_x})+s_{2,2}(\theta_4)\Delta y^{\beta}{\bf I}_{J},
		\end{align*}
		which completes the proof.
	\end{proof}
	\end{appendix}
\end{document}